\documentclass{article}
\usepackage{amsmath}
\usepackage{amssymb}
\usepackage[margin= 1in]{geometry}
\usepackage[demo]{graphicx}
\usepackage{float}
\usepackage{color}
\usepackage[all]{xy}
\usepackage{amsthm}
\usepackage{bbm}
\usepackage{soul}
\usepackage[shortlabels]{enumitem}
\usepackage{multicol}
\usepackage[utf8]{inputenc}
\usepackage{tikz}
\usetikzlibrary{positioning}
\usetikzlibrary{decorations,arrows,shapes}
\usepackage{tikz-cd}
\usepackage[labelfont=bf]{caption}
\usepackage{mathdots}

\makeatletter
\newtheorem*{rep@theorem}{\rep@title}
\newcommand{\newreptheorem}[2]{%
\newenvironment{rep#1}[1]{%
 \def\rep@title{#2 \ref{##1}}%
 \begin{rep@theorem}}%
 {\end{rep@theorem}}}
\makeatother

\theoremstyle{plain}
\newtheorem{thm}{Theorem}[section]
\newtheorem{prop}[thm]{Proposition}
\newtheorem{lem}[thm]{Lemma}
\newtheorem{cor}[thm]{Corollary}

\theoremstyle{remark}

\theoremstyle{definition}
\newtheorem{rem}[thm]{Remark}
\newtheorem{dfn}[thm]{Definition}

\newcommand{\N}{\mathbb{N}}
\newcommand{\Z}{\mathbb{Z}}
\newcommand{\A}{\mathbb{A}}

\newcommand{\K}{\mathcal{K}}
\newcommand{\T}{\mathbb{T}}

\newcommand{\p}{\operatorname{par}}
\newcommand{\im}{\operatorname{im}}

\newcommand{\E}{\mathcal{E}}
\newcommand{\into}{\hookrightarrow}

\newcommand{\F}{\mathcal{F}}
\newcommand{\D}{\mathfrak{d}}
\renewcommand{\O}{\mathcal{O}}

\newcommand{\id}{\mathbbm{1}}
\newcommand{\Aut}{\text{Aut}}
\newcommand{\Span}{\operatorname{span}}
\definecolor{purple}{rgb}{.54,0,.54}
\definecolor{green}{HTML}{228B22}

\title{Refinement of Higher-Rank Graph Reduction}
\author{S. Joseph Lippert}
\date{}

\usepackage{sectsty}
\sectionfont{\centering}

\begin{document}
	
		\maketitle
		
		\begin{abstract}
		    Given a row-finite, source-free, graph of rank k, we extend the definition of reduction introduced by Eckhardt et al. This constitutes a large step forward in the extension of the geometric classification of finite directed graph $C^*$-algebras presented by Eilers et al. to higher-rank graph $C^*$-algebras. This new move acts as an inverse to delay, directly extends the previous version, and provides previously undocumented Morita classes of k-graphs. In pursuit of this extension, we formalize what constitutes a \textit{higher-rank graph move}. Specifically, we use this formalization as a bridge between the new geometric reasoning and the classical category theoretic construction. 
		\end{abstract}

        \section{Introduction}
            
            Recently, there have been an exciting number of advancements in the field of $C^*$-algebra classification by $K$-theoretic invariants. Specifically, the \textit{Elliott invariant} has proven to be necessary and sufficient for classification of simple unital $C^*$-algebras which have finite nuclear dimension and satisfy the Universal Coefficient Theorem from \cite{rose}. That is, given two such algebras $\A_1$ and $\A_2$, we know $\A_1\cong \A_2$ if and only if their Elliott  invariants are isomorphic \cite{quasidiag,finite-simple-amen-1,finite-simple-amen-2}. An impressive amount of effort has already been put towards extending this result to the non-unital case \cite{classification-KK-cont}. 
            
            Much effort has also been put into classifying $C^*$-algebras that are not simple. The most straightforward of these examples would be the Cuntz-Krieger algebras $\O_A$. If $A$ is not irreducible, then $\O_A$ is not a simple algebra \cite{cuntz}.  However, $\O_A$ can be classified by ordered filtered $K$-theory \cite{compclass}. The Cuntz-Krieger algebras are closely related to graph $C^*$-algebras. Because of this relationship, researchers have been able to interpret these classification results in a purely geometric context.
            
            In particular, the work done by Eilers, Restorff, Ruiz, and S{\o}rensen \cite{compclass} determined a complete list of $6$ Morita equivalence preserving graph moves. We call a move, \textbf{(M)}, taking $E$ to $E_M$, Morita equivalence preserving if $C^*(E_M)\sim_{ME} C^*(E)$. Further, given two graphs $E$ and $F$ such that $C^*(E)\sim_{ME} C^*(F)$  there exists a finite sequence of these Morita equivalence preserving moves and their inverses converting $E$ into $F$. In effect, \cite{compclass} showed that the list of $6$ moves completely classified graph $C^*$-algebras up to Morita equivalence.
            
            This result was undoubtedly a huge step forward for the field, but there are a number of $C^*$-algebras that cannot be realized as $C^*(E)$ for some $E$. In particular, $C^*$-algebras with torsion in their $K_1$ group cannot be expressed in this way (cf-\cite[Equation 3.3]{rae}). This was one of the motivations for the introduction of higher-rank graphs (sometimes called $k$-graphs) and their $C^*$-algebras by Kumjian and Pask \cite{kp}. These are natural analogues of graph $C^*$-algebras, but they have been shown to include some $C^*$-algebras with torsion in their $K_1$ group \cite{evans08}. Additionally, higher-rank graph $C^*$-algebras were crucial to the proof  that every UCT Kirchberg algebra has nuclear dimension 1 \cite{ruiz-sims-sorensen}. They relate to solutions of the Yang-Baxter equations \cite{yang, vdovina-DM-solns} and have provided key examples for noncommutative geometry \cite{pask-rennie-sims, hajac-sims}.
            
            Because of the importance of these higher-rank graph $C^*$-algebras, it is sensible to try and generalize the result from \cite{compclass}. Eckhardt, Fieldhouse, Gent, Gillaspy, Gonzales, and Pask  began the project of generalizing the moves of Eilers et al. to the realm of $k$-graphs \cite{efgggp}. This paper introduced four $k$-graph moves that preserve Morita equivalence (insplitting, delay, sink deletion,  and reduction). The move delay was first introduced by Bates and Pask \cite{delay} for directed graphs.  However, in \cite{compclass} it was observed that up to Morita equivalence reduction is a left inverse of delay and only one of those moves was necessary for complete classification of graph $C^*$-algebras. Ben Listhartke, in his PhD thesis work, continued with the project and extended the move outsplitting to $k$-graphs \cite{outsplit}. 

            This ground breaking work left refinements and open questions for the community. This paper is meant to address two. Firstly, this paper will introduce a general definition of a higher-rank graph move. Since it is relatively easy to change a directed graph and still obtain a well defined directed graph, such formal definitions were not necessary for \cite{compclass} and could be sidestepped with an ad hoc approach in \cite{efgggp}. A similar issue was tackled in the early days of this field by \cite{quasimorph1} with their introduction of $k$-morphs. Notably, $k$-morphs can be used to change a graph of rank $k$ by increasing the rank to $k+1$. This is markedly different from the geometric changes introduced in \cite{efgggp}. For this reason we create a definition of higher-rank graph move that encompasses $k$-morphs as well as those introduced by Eckhardt et al. 
            
            Secondly, this paper refines the move reduction, \textbf{(R)}, to rely on weaker hypotheses. We show that this is a well defined higher-rank graph move preserving Morita equivalence. We show that \textbf{(R)} respects the inherent monoidal category structures of higher-rank graphs and $C^*$-algebras. We then show that \textbf{(R)} is a left inverse to delay, \textbf{(D)}, just as it was in the directed graph case.  

            The structure of the paper is as follows. In Section 3, we define a higher-rank graph move. Roughly speaking this is a way of associating a pair $(\Lambda,w)$, $\Lambda$ a higher-rank graph and $w$ a vertex in $\Lambda$, satisfying a set of hypotheses, $H_M$, with a higher-rank graph $\Lambda_M$ such that there is a ``memory" of $\Lambda$. Section 4 discusses how we may construct a set of hypotheses for \textbf{(R)} in a way that will preserve the tensor product structure inherent to higher-rank graphs and $C^*$-algebras. In particular, we define neighborhoods about a vertex and incorporate them into our hypotheses. Section 5 defines the move \textbf{(R)} and demonstrates that it is well defined, preserves Morita equivalences, and geometrically encompasses the Morita equivalences implied by the monoidal category structure of $C^*$-algebras. Section 6 discusses in detail \textbf{(R)}'s relation to the moves that inspired it \textbf{(CR)} and \textbf{(D)}.
            
            It is the hope of the author that the softening of hypotheses outlined in this paper can be applied to improve the moves insplit and outsplit. These changes will hopefully pave the way to extending the Cuntz splice to higher-rank graphs.
            
            \textbf{Acknowledgements:} The author would like to thank his PhD advisor, Elizabeth Gillaspy, for many helpful discussions and invaluable guidance. He would also like to thank Luke Davis for bringing this problem to his attention. Lastly the author would like to thank the annonymous referee whose comments inspired Section \ref{moves} which added necessary coherence to the paper as a whole. This research was partially supported by NSF grant DMS-1800749 to Elizabeth Gillaspy.

        \section{Preliminaries}
        
        \textbf{Notation:} For this paper, we take $0\in \N$. We view $\N^k$ as a category with composition of morphisms given by addition. Note that the category is Abelian with a single object denoted $0$ or $(0,\dots,0)$. Additionally, we will refer to the standard basis of $\N^k$ as $E:=\{e_1,\dots,e_k\}$ (when necessary we will use $E_\ell$ as the standard basis of $\N^\ell$). Given a directed graph $(G^0,G^1,s_G,r_G)$, we denote the path category $G^*$. Additionally, all path composition will be carried out from right to left.
        
        \begin{dfn}
            \cite[Definitions 1.1]{kp}  Let $\Lambda$ be a countable category and $d:\Lambda \to \N^k$ a functor. If $(\Lambda, d)$ satisfies the \textit{factorization property}- that is, for every morphism $\lambda \in \Lambda$ and $n,m\in \N^k$ such that $d(\lambda)=m+n$, there exist unique $\mu,\nu\in \Lambda$ such that $d(\mu)=m$, $d(\nu)=n$, and $\lambda = \mu\nu$- then $(\Lambda,d)$ is a \textit{$k$-graph} (or graph of rank $k$).
        \end{dfn}
        
        A helpful source of examples is product graphs. 
        
        \begin{prop}
        \label{prod}
        \cite[Proposition 1.8]{kp} Let $\Lambda_1$ and $\Lambda_2$ be rank $k$ and $\ell$ graphs respectively. Then the product category $\Lambda_1\times \Lambda_2$ with the product functor $d_1\times d_2(\lambda_1,\lambda_2)=(d_1(\lambda_1),d_2(\lambda_2))$ is a $k+\ell$-graph.
        \end{prop}
        
        \begin{dfn}
            \label{kghom}
            A functor $\phi: (\Lambda_1, d_1) \to (\Lambda_2, d_2)$ is a \textit{higher-rank graph morphism} if given composable pair $\lambda_1,\lambda_2\in \Lambda_1$ we have $\phi(\lambda_2\lambda_1)=\phi(\lambda_2)\phi(\lambda_1)$ and $d_2(\phi(\lambda_1))=d_1(\lambda_1)$. If $\phi$ is bijective, then $\phi$ is called an isomorphism and we write $\Lambda_1\cong \Lambda_2$.
        \end{dfn}
        
        \begin{rem}
            Define the natural projection $\pi_1: \Lambda_1\times\Lambda_2 \to \Lambda_1$. Note that $\pi_1$ is a higher-rank graph morphism by construction. 
        \end{rem}
        
        For a set $A\subseteq \N^k$ we will define $\Lambda^A = \{ \lambda \in \Lambda : d(\lambda)\in A\}$. When using singleton sets like $\{e_1\}$ we will suppress the $\{\cdot\}$. That is, $\Lambda^{e_1} := \Lambda^{\{e_1\}}$. We will regard $\Lambda^0$ as the set of \textit{vertices} of $\Lambda$. In particular, the factorization property guarantees that for any $\lambda\in \Lambda$ there exist  unique $v,w\in \Lambda^0$ such that $w\lambda v=\lambda$. We use the convention $s(\lambda)=v$ and $r(\lambda)=w$. This will allow us to define two more sets:
        \begin{center}
            \begin{tabular}{ccc}
                $w\Lambda = \{ \lambda\in \Lambda : r(\lambda)=w \}$ & and & $\Lambda v = \{ \lambda\in \Lambda : s(\lambda)=v \}$  \\
            \end{tabular}.
        \end{center}
        We define $w\Lambda^A$ and $\Lambda^A v$ analogously.
        
        \begin{dfn}
            Adopting the convention, $E:=\{e_1,\dots,e_k\}$ the standard basis of $\N^k$. We say that a $k$-graph, $\Lambda$, is \textit{row-finite} if for all $w\in \Lambda^0$ we have $|w\Lambda^E|<\infty$. Additionally we say that $\Lambda$ is \textit{source free} if $|w\Lambda^E|\neq 0$ for all $w\in \Lambda^0$.
        \end{dfn}
        
        \begin{dfn}
            For a $k$-graph, $\Lambda$, each morphism in $\Lambda^E$ has a $s$ or $r$ in $\Lambda^0$. We define the \textit{$1$-skeleton} as the directed graph $G(\Lambda)=(\Lambda^0,\Lambda^E,s,r)$. This can be thought of as a colored digraph with each edge in $f\in\Lambda^E$ having the associated color $d(f)\in E$ (the standard basis of $\N^k$).
        \end{dfn}

        Let $\mathbb{F}^+_k$ be the free semigroup on generators, $E$. Using the notation introduced in \cite{oneskel}, we extend $d|_{G(\Lambda)}$ to the functor $c:G(\Lambda)^*\to\mathbb{F}_k^+$. This functor is called the color order. There is also a canonical quotient map $\Pi: G(\Lambda)^* \to \Lambda$. We then define an equivalence relation on $G(\Lambda)^*$, $\sim_\Lambda$, such that $\lambda_1\sim_\Lambda \lambda_2$ if and only if $\Pi(\lambda_1)=\Pi(\lambda_2)$. 

        \begin{prop}
            \label{permprop}
            \cite[Proposition 4.7]{oneskel} Let $\Lambda$ be a $k$-graph with one skeleton $G(\Lambda)$. If $\lambda,\eta \in G(\Lambda)^*$ such that $\lambda\sim_\Lambda\eta$ and $c(\lambda)=c(\eta)$ then $\lambda=\eta$. Further, there is a canonical action of $S_{\|d(\lambda)\|_1}$ on $c(\lambda)$ and for every $\sigma\in S_{\|d(\lambda)\|_1}$ there exists a unique $\eta\in G(\Lambda)^*$ such that $\lambda\sim_\Lambda \eta$ and $c(\eta)=\sigma\cdot c(\lambda)$.
        \end{prop}
        
        The main result of \cite{oneskel} is the following
        \begin{prop}
        
\label{kg-rel-prop}
    \cite[Theorem 4.4]{oneskel}
    Let $(G,d)$ be a $k$-colored digraph. Let $\sim$ be an equivalence relation on $G^*$ satisfying:
    \begin{enumerate}[start=0,label={(KG\arabic*):}]
        \item If $\lambda\in G^*$ is a path such that $\lambda=\lambda_2\lambda_1$, then $[\lambda] = [p_2p_1]$ whenever $p_1\in [\lambda_1]$ and $p_2\in [\lambda_2]$.
        \item If $\lambda\in G^1$ then $\lambda\sim_\Lambda \mu$ implies $\lambda=\mu$.
        \item If $\lambda,\mu,\in G^1$ such that $s(\lambda)=r(\mu)$ and $d(\lambda)\neq d(\mu)$, then there exist a unique pair $\lambda',\mu'\in \Lambda^E$ such that $s(\lambda')=r(\mu')$, $d(\lambda)=d(\mu')$, $d(\mu)=d(\lambda')$, and $\lambda'\mu'\sim_\Lambda \lambda\mu$.
        \item Given a triplet $\lambda,\mu,\eta \in G^1$ such that $d(\lambda)\neq d(\mu) \neq d(\eta)$ the pairwise application of $\sim_\Lambda$ satisfies associativity. 
    \end{enumerate}
    The pair $(G^*/\sim,d)$ is a well-defined $k$-graph.
\end{prop}
    
        The upshot is that given a colored digraph $G$ with $k$ colors of edges, and an equivalence relation on $G^*$ that satisfies $(KG0-3)$, then we may construct a graph of rank $k$ out of those objects. Further, each graph of rank $k$ will produce a unique pair $G(\Lambda)^*$ and $\sim_\Lambda$ such that $G(\Lambda)^*/\sim_\Lambda\cong \Lambda$ \cite[Theorem 4.5]{oneskel}. So, we may move between these two schematics of a $k$-graph with little to no effort.
        \begin{rem}
        \label{prod-rem}
            Careful analysis of the product graph definition will reveal that $G(\Gamma\times\Lambda) = G(\Gamma)\square G(\Lambda)$, where $\square$ represents the Cartesian product of edge colored directed graphs. The factorization rule $\sim_\times$ is given by $(\gamma_1,\lambda_1)\sim_\times (\gamma_2,\lambda_2)$ if and only if $\gamma_1\sim_\Gamma \gamma_2 $ and $\lambda_1\sim_\Lambda \lambda_2$.
        \end{rem} 
        
        \begin{rem}
        \label{homrem}
            Note that a functor between path categories $\phi:G(\Lambda_1)^*\to G(\Lambda_2)^*$ such that $\mu\sim_1\nu$ implies $\phi(\mu)\sim_2\phi(\nu)$ will induce a higher-rank graph morphism $\hat{\phi}:\Lambda_1\to \Lambda_2$ such that $\hat{\phi}[\mu]_1=[\phi(\mu)]_2$. 
        \end{rem}
        
        Since the move described in this paper is meant to be interpreted geometrically, we will be using the $1$-skeleton picture almost exclusively. For this reason, we will take a moment to reintroduce the definition of a $k$-graph $C^*$-algebra in terms of the $1$-skeleton.
        
        \begin{dfn}
            \cite[Definition 1.5]{kp} Let $\Lambda$ be a row-finite, source free $k$-graph. A \textit{Cuntz-Krieger} (CK) $\Lambda$-\textit{family}, $\{T_\lambda: \lambda \in \Lambda^E\cup\Lambda^0\}$, is a set of partial isometries such that:
            \begin{enumerate}[start=0,label={(CK\arabic*):}]
                \item $\{ T_x:x\in \Lambda^0 \}$ is a set of mutually orthogonal projections.
                \item If $\lambda \mu\sim \eta\gamma$, then $T_{\lambda}T_{\mu} = T_\eta T_\gamma$.
                \item For any $\lambda\in \Lambda^E$, we have that $T^*_\lambda T_\lambda = T_{s(\lambda)}$.
                \item For any $x\in \Lambda^0$ and $e_i\in E$, we have that $\displaystyle{\sum_{\lambda\in x\Lambda^{e_i}}} T_\lambda T^*_\lambda = T_x$.
            \end{enumerate}
        \end{dfn}
        
        $C^*(\Lambda)$ is defined as the universal $C^*$-algebra generated by a Cuntz-Krieger $\Lambda$-family, that is $C^*(\Lambda)$ is generated by a Cuntz-Krieger $\Lambda$-family $\{s_\lambda:\lambda \in \Lambda^E\cup\Lambda^0\}$ with the property that for any CK $\Lambda$-family $\{T_\lambda: \lambda \in \Lambda^E\cup\Lambda^0\}$ there exists an onto $*$-homomorphism $\pi_{T}: C^*(\Lambda) \to C^*(T)$ such that $\pi_{T}(s_\lambda)=T_\lambda$.
        
        The final piece of the puzzle is that our work is concerned with preserving Morita equivalence between $C^*(\Lambda)$ and $C^*(\Lambda_R)$. To that end, we need to introduce some final results related to the general structure of $C^*(\Lambda)$. Our first result relates to the canonical gauge action $\alpha: \T^k \to \Aut(C^*(\Lambda))$ defined by its action on the generators, $\{s_\lambda:\lambda \in \Lambda^E\cup\Lambda^0\}$:
        \begin{align*}
            \alpha_z(s_\lambda) &= z^{d(\lambda)} s_\lambda 
        \end{align*}
        \noindent where, for $z=(z_1,\dots,z_k)\in \T^k$, we define $z^{(n_1,n_2,\dots,n_k)} := z_1^{n_1}z_2^{n_2}\cdots z_k^{n_k}$.
        
        \begin{thm}
            \label{gauge}
            \cite[Theorem 3.4]{kp} \textbf{(Guage-Invariant Uniqueness Theorem)} Fix a row-finite, source free $k$-graph, $\Lambda$, along with a $*$-homomorphism $\pi: C^*(\Lambda) \to B$.  If  $\pi(p_v)\neq 0$ for all $v$, and there exists an action $\beta: \T^k \to \Aut(B)$ such that 
            \[ \beta \pi = \pi \alpha, \]
            then $\pi$ is injective.
        \end{thm}
        
        Using Theorem \ref{gauge}, we obtain many notable isomorphisms of $k$-graph $C^*$-algebras. The one that is most useful for our purposes is the following.
        
        \begin{cor}
            \label{tensor}
            \cite[Corollary 3.5]{kp} For row finite source free higher-rank graphs, $\Lambda_1$ and $\Lambda_2$,
            \[ C^*(\Lambda_1 \times \Lambda_2) \cong C^*(\Lambda_1) \otimes C^*(\Lambda_2). \]
        \end{cor}

        Corollary \ref{tensor} demonstrates that the monoidal category structure of higher-rank graphs relates directly to the monoidal category structure of $C^*$-algebras. 
        
        Using Theorem \ref{tensor} effectively relies on constructing an action $\beta$. In constructing such actions we will utilize the following technical lemma.
        
        \begin{lem}\cite[Lemma 2.9]{efgggp}
        	\label{betalem}
        	Let $(\Lambda,d)$ be a row-finite source-free $k$-graph. Given a functor $\D: \Lambda \to \Z^k$, the function $\beta: \T^k \to \Aut(C^*(\Lambda))$ which satisfies
        	\[\beta_z(t_\mu t_\nu^*) = z^{\D(\mu)-\D(\nu)}t_\mu t_\nu^*\]
        	for all $\mu,\nu \in \Lambda$ and $z\in \T^k$, is an action of $\T^k$ on $C^*(\Lambda)$.
        \end{lem}

        Lemma \ref{betalem} will be the motivating tool behind the formal definition a higher-rank graph move. Until then, we will simply point out that this corollary defines gauge-like actions for functors to $\Z^k$.
        
        Finally, we need a way of understanding Morita equivalence of $k$-graph $C^*$-algebras. There is a well known result from \cite{allen} that we will restate here with added context to utilize later in our proofs. Specifically these results depend on the idea of a \textit{saturation} whose definition follows.
        
        \begin{dfn}
            Given a set $X\subseteq \Lambda^0$, we define its \textit{saturation}, $\Sigma(X)$, as the smallest set that contains $X$ and is:
            \begin{itemize}
                \item \textbf{Hereditary:} If $v\in \Sigma(X)$ and $\lambda\in v\Lambda$,
                then $s(\lambda)\in \Sigma(X)$.
                \item \textbf{Saturated:} If for some $n\in \N^k$ we have $s(v\Lambda^n)\subseteq \Sigma(X)$, then $v\in \Sigma(X)$.
            \end{itemize}
        \end{dfn}
        
        We now state a result that combines Remarks $3.2(2)$, Corollary $3.7$, and Proposition $4.2$ from \cite{allen}.
        
        \begin{thm}
        \label{morita}
            \cite[3.2(2), 3.7, 4.2]{allen} Let $\Lambda$ be a $k$-graph, $X\subseteq \Lambda^0$, and define $P_X = \displaystyle{\sum_{v\in X}} p_v$. If $\Sigma(X)=\Lambda^0$, then $P_XC^*(\Lambda)P_X \sim_{ME} C^*(\Lambda)$.
        \end{thm}

        \section{Higher-Rank Graph Moves}
        \label{moves}

        Thanks to the work of \cite{compclass,efgggp,outsplit} geometric transformations of higher-rank graphs have become a focus of recent research. Since this line of inquiry is still young, we will take some time to formalize certain concepts from a category theoretic perspective. In particular, we will offer new vocabulary that helps to unify this new work with classical higher-rank graph results.
        
    \begin{dfn}
        \cite[Defintion 2.2]{quasimorph1} For an $\ell$-graph, $\Gamma$, a $k$-graph, $\Lambda$, and a monoid morphism $\omega: \N^\ell \to \N^k$, a functor $\psi: \Gamma \to \Lambda$ is an $\omega$-quasimorphism if $\omega d_\Gamma = d_\Lambda \psi$.
    \end{dfn}

    For a $k$-graph, $\Lambda$, and an automorphism of $\N^k$, $\alpha$, there exists an $\alpha$-quasi-isomorphism from $(\Lambda,d)\to (\Lambda, \alpha d)$. For this reason, much liberty can be taken with the ordering of the standard basis vectors.

   \begin{dfn}
       Suppose that $\Omega$ is a small category and $(\Gamma,d_\Gamma)$ is an $\ell$-graph. The pair of small category morphisms $\varphi: \Gamma \to \Omega$ and $\D_\varphi: \Omega \to \Z^\ell$ is called a \textit{$\Omega$-realization of $\Gamma$} if $\D_\varphi\varphi=d_\Gamma$. Often for succinctness we will say ``let $\varphi$ be a $\Lambda$-realization" with the existence of $\D_\varphi$ implied.
   \end{dfn}

   The word \textit{realization} is chosen here since $(\im(\varphi),\D)$ always forms an $\ell$-graph. Further still, $\Gamma\cong \im(\varphi)$ for injective $\varphi$. 

    \begin{prop}
    \label{inj/quasi-prop}
        If $\omega: \N^\ell\to \N^k$ is an injective monoid morphism, then an $\omega$-quasimorphism between higher-rank graphs $\psi: \Gamma \to \Lambda$ induces a $\Lambda$-realization of $\Gamma$.
    \end{prop}

    \begin{proof}
        Since $\omega$ is injective, there exists a left inverse $\pi$ such that $\pi\omega=\id_{\N^\ell}$. In particular, the tuple $(\psi,\pi d_\Lambda)$ is a $\Lambda$-realization of $\Gamma$. That is, $(\pi d_{\Lambda})\psi = \pi \omega d_\Gamma = d_\Gamma$.
    \end{proof}

    Proposition \ref{inj/quasi-prop} ensures that higher-rank graph morphisms (which are $\id$-quasimorphisms \cite[Definition 1.1]{kp}) likewise induce higher-rank graph realizations.

    \begin{dfn}
        A higher-rank graph move \textbf{(M)} is a mapping from pairs $(\Lambda,w)$, with $\Lambda$ a higher rank graph and $w\in \Lambda^0$, satisfying a set of hypotheses, $H_M$, to the category of higher-rank graphs. The image of $(\Lambda,w)$ is referred to as $\Lambda_M$. Further there exists either a $\Lambda$-realization of $\Lambda_M$ or a $\Lambda_M$-realization of $\Lambda$. 
    \end{dfn}

    As put forth in the introduction, this definition is meant to generalize the notion of graph move from \cite{compclass}. This next proposition is meant to further emphasize this sentiment.

    \begin{prop}
        Directed graph reduction \textbf{(DR)} \cite[Proposition 3.2]{move-eq} defined on a digraph, $G$, with regular vertex $w$ such that $|r^{-1}(w)|=|r(s^{-1}(w))|=1$ produces the graph $G_{DR}$ 
        \begin{align*}
            G_{DR}^0 & = G^0\setminus\{w\} & G_{DR}^1 &= G^1\setminus(s^{-1}(w)\cup r^{-1}(w)) \cup \{[ef]: s(e)=w\} 
        \end{align*}
        with $s_{DR}$ and $r_{DR}$ induced by $s,r$ on $G^1\setminus(r^{-1}(w)\cup s^{-1}(w))$, $s_{DR}(ef)=s(f)$, and $r_{D}(ef)=r(e)$. \textbf{(DR)} is a graph move from $(G^*,w)$ to $G_R^*$ with $G^*$-realization of $G_R^*$.
    \end{prop}

    \begin{proof}
        Firstly, the set of hypotheses $H_R$ is the requirement of rank $1$, that $w$ is a regular vertex, and $|r^{-1}(w)|=|r(s^{-1}(w))|=1$. These hypotheses are on the tuple $(G^*,w)$. Next, $G_{DR}$ is constructed explicitly so we may form the path category $G_{DR}^*$. 

        What remains is to find a $G^*$-realization of $G_{DR}^*$. This realization is defined on edges of $G_{DR}^*$ and extended naturally to paths. On edges $e\in G^1\setminus(s^{-1}(w)\cup r^{-1}(w))$, we use natural inclusion $\varphi(e)=e$. For the added edges, $[ef]$, we define $\varphi([ef])=ef$. Since we have defined $\varphi: G_{DR}^1 \to G^*$, this extends to a functor $\varphi:G_{DR}^*\to G^*$.

        Finally, we need the connecting functor $\D:G^*\to \Z$. Again, we define $\D$ on $G^1$ and extend multiplicatively. $\D(e)=1$ for $r(e)\neq w$ and $\D(e)=0$ otherwise. This ensures that $\varphi \D = d_{DR}$ as desired.
    \end{proof}

    The moves from \cite{compclass} fall into two families. In this paper, we introduce vocabulary to distinguish these families.

    \begin{dfn}
        A higher-rank graph move, \textbf{(A)}, is called an \textit{adjustment} if for all pairs $(\Lambda,w)$ satisfying $H_A$ there exists a $\Lambda$-realization of $\Lambda_A$. We call this realization $\p_A: \Lambda_A \to \Lambda$, the \textit{parent function}.
    \end{dfn}

    In the directed graph case, the moves \textbf{(I),(O),(R),} and \textbf{(S)} are adjustments. Interestingly $\p_S/\p_R$ are injective and $\p_I/\p_O$ are surjective.

    \begin{dfn}
        A higher-rank graph move, \textbf{(Sp)}, is called a \textit{splice} if for all pairs $(\Lambda, w)$ satisfying $H_{Sp}$ there exists a $\Lambda_{Sp}$-realization of $\Lambda$, $\varphi_{Sp}: \Lambda \to \Lambda_{Sp}$.
    \end{dfn}

    In the directed graph case, the moves \textbf{(C)} and \textbf{(P)} are splices. Additionally, $k$-morphs which are defined in \cite{quasimorph1} are closely related to splices.

    \begin{prop}
        Fixing a $k$-graph, $\Gamma$, we may define the higher-rank graph splice, Link $\Gamma$. A pair $(\Lambda,w)$ satisfies $H_{L(\Gamma)}$ if there exists a $\Lambda-\Gamma$ morph, $X$. Link $\Gamma$ then sends $(\Lambda,w)$ to the linking $(k+1)$-graph $\Sigma$ \cite[Definition 4.3]{quasimorph1}.
    \end{prop}

    \begin{proof}
        This is an immediate consequence of the natural inclusion $\Lambda \into \Sigma$ established in \cite[\S 4]{quasimorph1}. 
    \end{proof}

    In the work of \cite{compclass}, directed graph moves were defined on the set of vertices and edges. In an effort to mirror this construction, \cite{efgggp} defined their higher-rank graph moves in terms of the $1$-skeleton. The following lemma is meant to better understand how a higher-rank graph move can be constructed geometrically from $1$-skeletons.

    \begin{lem}
    \label{1-skel-realiz}
        Let $G$ be a directed graph with an edge coloring in $\ell$ colors, $d_G: G^1\to E_\ell$. Let $(\Lambda,d_\Lambda)$ be a $k$-graph with functors, $\varphi: G^1 \to \Lambda$ and $\D: \Lambda \to \Z^\ell$ such that $R\varphi = d_G$ on $G^1$. If $(\im(\varphi),\sim, R)$ satisfies $(KG2-3)$ then there exists an equivalence relation, $\sim_\varphi$ such that $G^*/\sim_\varphi$ is a well-defined $\ell$-graph. 
    \end{lem}

    \begin{proof}
        Since our intention is to transform $G$ into a higher-rank graph, we must construct $\sim_\varphi$ such that it satisfies the $KG$ conditions. First, we define $\sim_\varphi$ on $G^0$ and $G^1$ as $f\sim_\varphi g$ if and only if $f=g$. For paths $\lambda,\nu\in G^n$ for $n\geq 2$, we define $\lambda\sim_\varphi \nu$ if and only if $\varphi(\lambda)= \varphi(\nu)$. By construction, $\sim_\varphi$ satisfies $(KG1)$. Additionally under this construction for all $\lambda,\nu\in G^*$ $\lambda\sim_\varphi \nu$ implies $\varphi(\lambda)=\varphi(\nu)$.

        To show $(KG0)$, let $\lambda_1,\lambda_2,p_1,p_2\in G^*$, $p_1\sim_{\varphi}\lambda_1$, and $p_2\sim_{\varphi}\lambda_2$. Consider $\varphi(\lambda_1\lambda_2)$. Since $\varphi$ is a functor, we obtain 
        \[ \varphi(\lambda_1\lambda_2) = \varphi(\lambda_1)\varphi(\lambda_2) = \varphi(p_1)\varphi(p_2)=\varphi(p_1p_2).\]

        Finally, by hypothesis $\sim_\varphi$ satisfies $(KG2-3)$. We conclude that $G^*/\sim_{\varphi}$ is an $\ell$-graph. Lastly, this construction induces a natural $\Lambda$-realization of $G^*/\sim_\varphi$
    \end{proof}

    With Lemma \ref{1-skel-realiz}, we obtain a methodology for constructing higher-rank graph adjustment. Specifically, we begin with a $k$-graph and alter the $1$-skeleton to obtain a new colored digraph for which there is a natural parent functor. This natural parent function then induces a factorization rule which transforms the colored digraph into the adjusted higher-rank graph. While this methodology was not made explicit in \cite{efgggp} if the reader were to revisit this work they will find that this was used to construct moves and demonstrate that the they are well-defined.

    The moves used in \cite{compclass} and those introduced in \cite{efgggp} were of interest because of their affect on $C^*$-algebras. For this reason, we introduce vocabulary relating higher-rank graph moves to $C^*$-algebraic properties.

    \begin{dfn}
        A higher-rank graph move \textbf{(M)} is called \textit{subtle} if for all pairs $(\Lambda,w)$ satisfying the hypotheses $H_M$ and $\Lambda$ source-free and row-finite there is Morita-equivalence $C^*(\Lambda)\sim_{ME} C^*(\Lambda_M)$.
    \end{dfn}

        In the case of $1$-graphs, it was shown that the $6$ moves of \cite{compclass} were all subtle moves. Additionally, all moves introduced in \cite{efgggp} were shown to be subtle. 

    \begin{prop}
        \label{varphi-prop}
            Suppose that $\Gamma$ and $\Lambda$ are $\ell$ and $k$-graphs respectively. Let $\varphi$ be a $\Lambda$-realization of $\Gamma$, and let $\{t_{\gamma}: \gamma\in \Gamma\}$ and $\{s_{\lambda}: \lambda\in \Lambda\}$ be the generators of $C^*(\Gamma)$ and $C^*(\Lambda)$ respectively. If $\{s_{\varphi(\lambda)}: \lambda\in \Lambda\}$ is a CK $\Gamma$-family, then there exists an injective $*$-homomorphism which we call $\hat{\varphi}: C^*(\Gamma)\into C^*(\Lambda)$ such that $t_{\gamma}\mapsto s_{\varphi(\gamma)}$.
        \end{prop}

        \begin{proof}
            By hypothesis, $\{s_{\varphi(\gamma)}: \gamma\in \Gamma\}$ is a CK $\Gamma$-family. In particular, the universal property of $C^*(\Gamma)$ gives a $*$-homomorphism $\hat{\varphi}: C^*(\Gamma) \to C^*(\Lambda)$ such that $t_{\gamma}\mapsto s_{\varphi(\gamma)}$. Since $\varphi$ is a functor it sends vertices to vertices, in particular, $\hat{\varphi}$ sends no vertex projections to $0$. Let $\beta: \T^k\to \Aut(C^*(\Lambda))$ as defined in Lemma \ref{betalem} with $\D_\varphi$ and observe that
            \[ \beta_z(s_{\varphi(\mu)}s_{\varphi(\nu)}^*) = z^{\D_\varphi \varphi (\mu)-\D_\varphi \varphi (\nu)} s_{\varphi(\mu)}s_{\varphi(\nu)}^* = z^{d_\Gamma(\mu)-d_\Gamma(\nu)} \hat\varphi(s_{\mu}s_{\nu}^*).\]
            By Theorem \ref{gauge} we conclude that $\hat{\varphi}$ is injective.
        \end{proof}

        Proposition \ref{varphi-prop} yields a clear methodology for demonstrating that an adjustment, \textbf{(A)}, is subtle. Since \textbf{(A)} comes equipped with a $\Lambda$-Realization of $\Lambda_A$, it is only necessary to show that the parent function produces a CK $\Gamma$-function then apply Proposition \ref{morita}.

        As discussed in the introduction, both the category of higher-rank graphs and the category of $C^*$-algebras are monoidal. Thus we need a way of discussing how this information is preserved by graph moves.
        \begin{thm}
            \cite[Theorem 1.2]{morita-stab} \textbf{(Brown-Green-Rieffel Theorem)} Let $\A_1$ and $\A_2$ be $C^*$-algebras with countable approximate identities. Then $\A_1\otimes\K\cong \A_2\otimes\K$ if and only if $\A_1$ is Morita equivalent to $\A_2$ (written $\A_1\sim_{ME}\A_2$).
        \end{thm}

        Higher-rank graph $C^*$-algebras are countably generated and thus have a countable approximate identity. This yields the following corollary.
        
        \begin{cor}
        \label{morita-rob}
            Let $\Omega$, $\Lambda$, and $\Gamma$ be source free, row-finite, $j$, $k$, and $\ell$-graphs respectively. If $C^*(\Lambda)\sim_{ME} C^*(\Gamma)$, then $C^*(\Lambda \times \Omega)\sim_{ME} C^*(\Gamma\times\Omega)$
        \end{cor}

        \begin{proof}
            As shown in \cite[Corollary 3.5]{kp}, $C^*(\Lambda\times\Omega)\cong C^*(\Lambda)\otimes C^*(\Omega)$. This coupled with the Brown-Green-Rieffel Theorem and the commutativity/associativity of tensor products gives:
            \[ C^*(\Lambda \times \Omega)\otimes \K \cong (C^*(\Lambda)\otimes\K) \otimes \Omega \cong (C^*(\Gamma)\otimes\K) \otimes \Omega \cong C^*(\Gamma \times \Omega)\otimes \K.\qedhere \]
        \end{proof}

        In short, the corollary above demonstrates that there is a second helpful property for graph moves to have. We define this property as robustness.
        
        \begin{dfn}
            \label{robdef}
            A higher-rank adjustment/splice, \textbf{(M)} is said to be \textit{robust} if for any pair $(\Lambda, w)$ satisfying $H_M$ and any $j$-graph $\Omega$ the product graph satisfies $H_M$ and there exists an isomorphism making either of these diagrams commute (left for adjustments right for splices).
\begin{center}

            \begin{tikzcd}
                \Lambda \times \Omega & \ar[l] \Lambda_M \times \Omega & & \Lambda \times \Omega \ar[r] \ar[d] & \Lambda_M \times \Omega\\
                (\Lambda\times\Omega)_M\ar[u]\ar[<->,ur] & & & (\Lambda\times\Omega)_M \ar[<->,ur] & 
            \end{tikzcd}
            \end{center}
        \end{dfn}

        Notably, robustness is a property unique to higher-rank graph moves. Specifically, taking products increases the rank of the graph so the hypothesis set must accommodate all graph ranks.  It can be checked that of the higher-rank graph moves in \cite{efgggp} only the moves \textbf{(D)} and \textbf{(S)} are robust. This is a desirable property for higher-rank graph moves preserving Morita equivalence to have. Specifically, it demonstrates that the move does not imply product graph Morita equivalences that it fails to account for.

	\section{Neighborhoods of a Vertex and Reducibility}

        A simple way to ensure robustness of a higher rank graph move is to define the hypothesis set, $H_M$, with respect to only a subset of the $k$ colors in the graph. We call this a color set, $B\subseteq E$. This will ensure that after the basis grows when taking a product $H_M$ is still satisfied. 

        In this section, we will lay the groundwork necessary to define hypothesis sets in this way.  In particular, we will start from a pair $(\Lambda, w)$ and bifurcate the $1$-skeleton into $\Lambda^B$ and $\Lambda^\E$ (where $\E:=E\setminus B$). We then adjust these pieces individually with the goal of eventually recombining them into our adjusted higher-rank graph. 
\begin{center}

        \begin{tikzcd}
            \Lambda^E \ar[r]\ar[dr] & \Lambda^\E \ar[r] & (\Lambda^\E)' \ar[rd]& \\
             & \Lambda^B \ar[r] & \Lambda_M^B \ar[r] & \Lambda_M^E
        \end{tikzcd}
            
\end{center}

The focus of this section is to understand the particulars of this diagram. Our eventual goal is to define a graph move reduction for pairs $(\Lambda,w)$. Just as in \cite{move-eq}, we will need to establish criteria that make a vertex $w$ reducible. This will be the start of our hypothesis set. 

	\begin{dfn}
	\label{red}
		We say that $w$ is \textit{reducible} with color set $\emptyset\neq B\subseteq E$ if there exists $v\in \Lambda^0$ such that the following hold:
		\begin{align*}
		    \mu\in w\Lambda^E v &\implies d(\mu)\in B &\forall b\in B, \hspace{2mm} w\Lambda^b&=\{f_b\}\\
			w& \notin r(\Lambda^Bw) & 			s(w\Lambda^B) &= \{v\}.
		\end{align*}
        We call $w\Lambda^B$ the set of bridge edges and $\Lambda^B w$ the set of co-bridge edges.
	\end{dfn}

    Classically, the move reduction was used to remove a reducible vertex and its bridge edges. However, in our current state we cannot easily remove $w$. If we did, it would make recombining impossible. To account for this, we must change $\Lambda^\E$.
	
	\begin{dfn}
	    For a $k$-graph, $\Lambda$, a vertex $w\in \Lambda^0$, and a color set, $\E\subseteq E$, we define the neighborhood about $w$ in color set $\E$,  $U_w^\E$, as the connected component of $\Lambda^\E$ containing $w$. 
	\end{dfn}
	
	\begin{figure}[H]
	\centering
	\stepcounter{thm}
	\scalebox{.6}{
	\begin{tikzpicture}[roundnode/.style={circle, draw=black!60, fill=gray!5, very thick, minimum size=3mm}]
            \node[roundnode]        (w)                             {y};
            \node[roundnode]        (x)       [below=4 of w]  {v};
            \node[roundnode]        (v)       [right= 4 of w]  {x};
            \node[roundnode]        (y)     [right=4 of x]      {w};
            \node[]                 (ur)    [above right= .75 of v]  {};
            \node[]                 (ul)    [above left= .75 of w]  {};
            \node[]                 (dr)    [below right= .75 of y]  {};
            \node[]                 (dl)    [below left= .75 of x]  {};

            \node[roundnode]        (w2)    [right = 4 of v]   {y};
            \node[roundnode]        (x2)       [below=4 of w2]  {v};
            \node[roundnode]        (v2)       [right= 4 of w2]  {x};
            \node[roundnode]        (y2)     [right=4 of x2]      {w};
            \node[]                 (ur2)    [above right= .75 of v2]  {};
            \node[]                 (ul2)    [above left= .75 of w2]  {};
            \node[]                 (dr2)    [below right= .75 of y2]  {};
            \node[]                 (dl2)    [below left= .75 of x2]  {};
            
            \node[]                 (u1)    [above right = 2 and 2 of v] {};
            \node[]                 (d1)    [below right = 2 and 2 of y] {};

            \node[scale=2]                         [above right = 2 and 2 of w] {$\Lambda^E$};
            \node[scale=2]                         [above right = 2 and 1 of w2] {$\Lambda^\E$};
            
            \draw[->, line width = 1.5, draw = blue, dotted] (v.west) to[bend left = 10] (w.east);
            \draw[->, line width = 1.5, draw=black] (v.west) to[bend right = 10] (w.east);
            
            \draw[->, line width = 1.5, draw = blue, dotted] (w.south) to[bend left = 10] (x.north);
            \draw[->, line width = 1.5, draw=black] (w.south) to[bend right = 10] (x.north);
            
            \draw[->, line width = 1.5, draw = blue, dotted] (x.east) to[bend left = 10] (y.west);
            \draw[->, line width = 1.5, draw=black] (x.east) to[bend right = 10] (y.west);
            
            \draw[->, line width = 1.5, draw = blue, dotted] (y.north) to[bend left = 10] (v.south);
            \draw[->, line width = 1.5, draw=black] (y.north) to[bend right = 10] (v.south);
            
            \draw[->, line width = 1.5, draw = red, dashed] (v.east) to[ in=315,bend right = 70] (ur) to[out=135,bend right = 70] (v.north);
            \draw[->, line width = 1.5, draw = red, dashed] (x.west) to[in=135, bend right = 70] (dl) to[out=315, bend right=70] (x.south);
            \draw[->, line width = 1.5, draw = red, dashed] (y.south) to[in= 225, bend right = 70] (dr) to[out=45, bend right=70] (y.east);
            \draw[->, line width = 1.5, draw = red, dashed] (w.north) to[in=45, bend right = 70] (ul) to[out=225, bend right=70] (w.west);
            
            \draw[->, line width = 1.5, draw = red, dashed] (x2.west) to[in=135, bend right = 70] (dl2) to[out=315, bend right=70] (x2.south);
            \draw[->, line width = 1.5, draw = red, dashed] (y2.south) to[in= 225, bend right = 70] (dr2) to[out=45, bend right=70] (y2.east);
            \draw[->, line width = 1.5, draw = red, dashed] (w2.north) to[in=45, bend right = 70] (ul2) to[out=225, bend right=70] (w2.west);
            \draw[->, line width = 1.5, draw = red, dashed] (v2.east) to[ in=315,bend right = 70] (ur2) to[out=135,bend right = 70] (v2.north);

            \draw[-,dotted, line width=2] (u1) to[] (d1);
            
            \draw[color=green, dotted, line width=4] (y.south east) circle (1);
            
            \node[scale=2, color=green]              [below left = .5 of dr] {$U_w^\E$};
            
            \draw[color=green, dotted, line width=4] (y2.south east) circle (1);
            
            \end{tikzpicture}
        }
    	\caption{In the above figure, $\Lambda$ is a $3$-graph with its $1$-skeleton pictured, and we take $\E$ to be the color red (dashed). }
    	\label{sqfig}
    \end{figure}
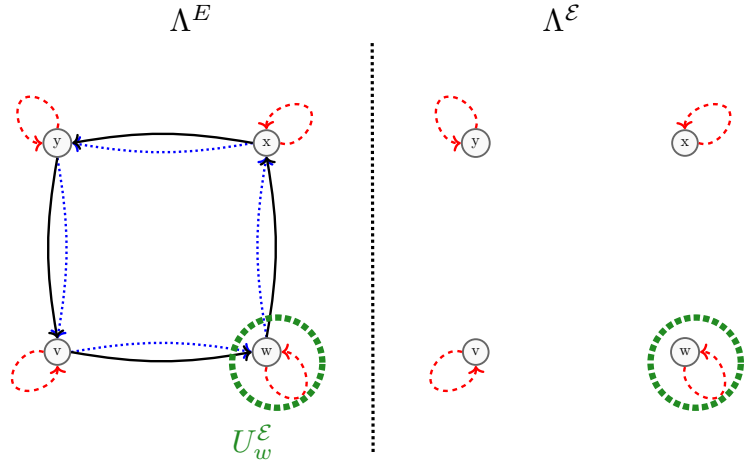

    From this definition, we have found a compelling candidate in $(\Lambda^\E)'=\Lambda^\E\setminus U_w^\E$. This however provokes us to further change $\Lambda^B$ to account for the possibly missing vertices.

    \begin{dfn}
    \label{full-red-def}
	    We say that $w$ is \textit{fully reducible} with color set $B$ if after defining $\E:=E\setminus B$ each $x\in(U_w^\E)^0$ is reducible with color set $B$ and 
     \[ \left(\bigcup_{x\in (U_w^\E)}\Lambda^Bx \right) \cap \left( \displaystyle{\bigcup_{x\in (U_w^\E)}}x\Lambda^B \right) = \emptyset \]
      i.e. the bridge edges and co-bridge edges are disjoint. For any $x\in (U_w^\E)^0$, we utilize the notation $f_{b,x}$ for the unique bridge edge of color $b\in B$ such that $r(f_{b,x})=x$.
	\end{dfn}

    The disjoint condition can be thought of as an extension of the exclusion of loops in Definition \ref{red}. In particular, the following theorem demonstrates that the only way into $U_w^\E$ is a bridge edge and the only way out is a co-bridge edge.

 \begin{prop}
 \label{prop-in-out}
     Let $w$ be fully reducible. If $\theta$ is a co-bridge edge, then $r(\theta)\notin U_w^\E$.
 \end{prop}

 \begin{proof}
     Suppose $r(\theta)\in U_w^\E$. By hypothesis, $r(\theta)$ is reducible with color set $B$. Since $\theta \in r(\theta)\Lambda^B$, we conclude that $\theta$ is a bridge edge, a contradiction.
 \end{proof}

Returning to Figure \ref{sqfig}, we notice that $w$ is a reducible vertex with color set black (solid) and blue (dotted). Further, with $\E$ taken as red (dashed), $(U_w^{\E})^0=\{w\}$, thus $w$ is fully reducible. This example is of particular note because it is the $1$-skeleton of a product graph.

\section{The Move Reduction}
    Our current schematic for the higher-rank graph adjustment, \textbf{(R)}, is represented in the figure below.

    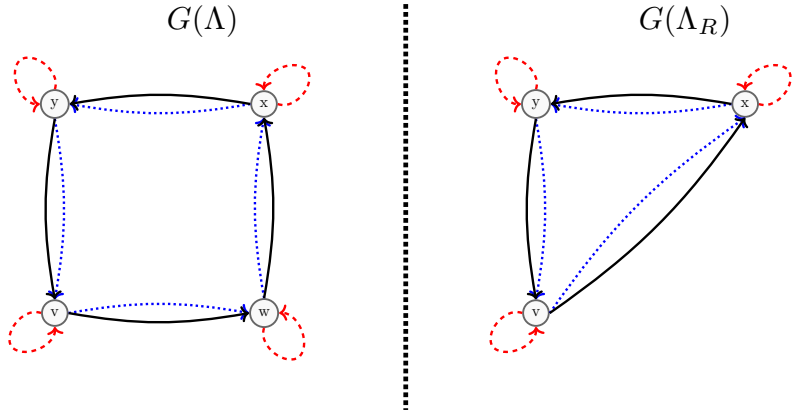
\begin{figure}[H]
    \centering
    \stepcounter{thm}
    
    \scalebox{.6}{
    \begin{tikzpicture}[roundnode/.style={circle, draw=black!60, fill=gray!5, very thick, minimum size=3mm}]
        \node[roundnode]        (w)                             {y};
        \node[roundnode]        (x)       [below=4 of w]  {v};
        \node[roundnode]        (v)       [right= 4 of w]  {x};
        \node[roundnode]        (y)     [right=4 of x]      {w};
        \node[]                 (ur)    [above right= .75 of v]  {};
        \node[]                 (ul)    [above left= .75 of w]  {};
        \node[]                 (dr)    [below right= .75 of y]  {};
        \node[]                 (dl)    [below left= .75 of x]  {};
        
        \draw[->, line width = 1.5, draw = blue, dotted] (v.west) to[bend left = 10] (w.east);
        \draw[->, line width = 1.5, draw=black] (v.west) to[bend right = 10] (w.east);
        
        \draw[->, line width = 1.5, draw = blue, dotted] (w.south) to[bend left = 10] (x.north);
        \draw[->, line width = 1.5, draw=black] (w.south) to[bend right = 10] (x.north);
        
        \draw[->, line width = 1.5, draw = blue, dotted] (x.east) to[bend left = 10] (y.west);
        \draw[->, line width = 1.5, draw=black] (x.east) to[bend right = 10] (y.west);
        
        \draw[->, line width = 1.5, draw = blue, dotted] (y.north) to[bend left = 10] (v.south);
        \draw[->, line width = 1.5, draw=black] (y.north) to[bend right = 10] (v.south);
        
        \draw[->, line width = 1.5, draw = red, dashed] (v.east) to[ in=315,bend right = 70] (ur) to[out=135,bend right = 70] (v.north);
        \draw[->, line width = 1.5, draw = red, dashed] (x.west) to[in=135, bend right = 70] (dl) to[out=315, bend right=70] (x.south);
        \draw[->, line width = 1.5, draw = red, dashed] (y.south) to[in= 225, bend right = 70] (dr) to[out=45, bend right=70] (y.east);
        \draw[->, line width = 1.5, draw = red, dashed] (w.north) to[in=45, bend right = 70] (ul) to[out=225, bend right=70] (w.west);
        
        \node[roundnode]        (w1)       [right=10 of w]                      {y};
        \node[roundnode]        (x1)       [below=4 of w1]  {v};
        \node[roundnode]        (v1)       [right= 4 of w1]  {x};
        \node[]                 (ur1)    [above right= .75 of v1]  {};
        \node[]                 (ul1)    [above left= .75 of w1]  {};
        \node[]                 (dl1)    [below left= .75 of x1]  {};
        
        \draw[->, line width = 1.5, draw = blue, dotted] (v1.west) to[bend left = 10] (w1.east);
        \draw[->, line width = 1.5, draw=black] (v1.west) to[bend right = 10] (w1.east);
        
        \draw[->, line width = 1.5, draw = blue, dotted] (w1.south) to[bend left = 10] (x1.north);
        \draw[->, line width = 1.5, draw=black] (w1.south) to[bend right = 10] (x1.north);
        
        \draw[->, line width = 1.5, draw = blue, dotted] (x1.east) to[bend left = 10] (v1.south);
        \draw[->, line width = 1.5, draw=black] (x1.east) to[bend right = 10] (v1.south);

        \draw[->, line width = 1.5, draw = red, dashed] (v1.east) to[ in=315,bend right = 70] (ur1) to[out=135,bend right = 70] (v1.north);
        \draw[->, line width = 1.5, draw = red, dashed] (x1.west) to[in=135, bend right = 70] (dl1) to[out=315, bend right=70] (x1.south);
        \draw[->, line width = 1.5, draw = red, dashed] (w1.north) to[in=45, bend right = 70] (ul1) to[out=225, bend right=70] (w1.west);
        
        \node[]     (u)             [above right = 2 and 2.8 of v] {};
        
        \node[]     (d)             [below = 9 of u] {};
        
        \draw[dotted, line width = 3] (u) to (d);
        
        \node[scale=2]      [above right = 1 and 2 of w] {$G(\Lambda)$};
        
        \node[scale=2]      [above right = 1 and 1.8 of w1] {$G(\Lambda_R)$};
        
        \end{tikzpicture}
    }
    
    \caption{Recall from earlier that $w$ was fully reducible with color set blue (dotted) and black (solid). For this reason we proceed by deleting $U_w^\E$ and reducing in $\Lambda^B$ at each of the deleted vertices.}
    \label{redex}
\end{figure}

    Since we have a way of associating an ordered pair $(\Lambda,w)$ with a one skeleton, we may use Proposition \ref{1-skel-realiz} to obtain an adjustment. Specifically, we must define $\D:\Lambda\to \Z^k$ and $\p:G(\Lambda_R) \to \Lambda$ and show that $\im(\p)$ satisfies $(KG2-3)$. 
    
\begin{dfn}
	\label{reduction}
	\textbf{Reduction (R):} Let $\Lambda$ be a $k$-graph. Suppose that $w$ is fully reducible with color set $B$, and suppose that for any $x\in (U_w^\E)^0$ the set $x\Lambda^B$ is a stationary set with color set $B$ \text{(Definition \ref{stat})}. Define the colored digraph $G(\Lambda_R)$ as follows:
	\begin{align*}
		\Lambda_R^0 & = \Lambda^0 \setminus (U_w^\E)^0\\
		\Lambda_R^E & = \Lambda^E \setminus((U_w^\E)^E \cup \{r_B^{-1}(x): x\in (U_{w}^\E)^0\})\\
		r_R(\mu) & =  r(\mu)\\
		s_R(\mu) & = \begin{cases}
			s(f_{b,s(\mu)}), & \text{ if } s(\mu)\in  U_{w}^\E \\
			s(\mu), & \text{ otherwise. }
		\end{cases}
	\end{align*}

	Note that in this construction each edge in $\Lambda_R^E$ shares a name with an edge in $\Lambda$. Thus we may define the natural inclusion $\iota: \Lambda_R^E \to \Lambda^E$, and when necessary this may be extended multiplicatively. Fix a $b\in B$, define $\p: \Lambda_R^E \to \Lambda$:
	\[ \p(\mu)=\begin{cases}
		\iota(\mu) f_{b,s(\iota(\mu))}, & \text{ if } s(\iota(\mu))\in (U_w^\E)^0 \\
		\iota(\mu), & \text{ otherwise. }
	\end{cases} \]
	Extend $\p$ to paths multiplicatively. 

    Define $\D:\Lambda \to \Z^k$:
    \begin{align*}
		\D(\mu)&=d(\mu) \text{, for } \mu\in \Lambda^E \text{ s.t. } s(\mu) \notin U_w^\E \text{ or } d(\mu)\notin B \\
		\D(\mu)& = d(\mu) - b \text{, for } \mu\in \Lambda^E \text{ s.t. } s(\mu) \in U_w^\E \text{ and } d(\mu)\in B \\
		\D(\lambda) & = \sum_{i=1}^{|\lambda|}\D(\lambda_i) \text{, for } \lambda=\lambda_{|\lambda|}\dots\lambda_1\in G(\Lambda)^*\\
		\D(x) & = 0 \text{, for } x\in \Lambda^0.
	\end{align*} 
\end{dfn}

We now present the main theorem of the paper which we will prove at the end of this section.
\begin{thm}
\label{main-thm}
    The graph move, \textbf{(R)}, is a well-defined, subtle, higher-rank graph adjustment. 
\end{thm}

We will prove this in two parts. First, we will demonstrate well-definedness of \textbf{(R)} this involves showing that $\D:\Lambda \to \Z^k$ is well-defined and showing that $\im(\p)$ satisfies $(KG2-3)$. From here, Proposition \ref{1-skel-realiz} will give a $\Lambda$-realization of $\Lambda_R$ guaranteeing that \textbf{(R)} is an adjustment. Second we will demonstrate that \textbf{(R)} is subtle with the aid of Proposition \ref{varphi-prop}.

Well-definedness of $\D$ and the satisfaction of $(KG2-3)$ is dependent on the new addition to $H_R$, the requirement that bridge edges be stationary. This definition is an extension of the \textit{complete edge} \cite[Notaion 6.1]{efgggp}.

\begin{dfn}
	For a higher-rank graph $\Lambda$, a set of edges $\mathcal{F}\subseteq \Lambda^E$ is called \textit{stationary} in color set $B\subseteq E$ if for all $f_1,f_2\in \F$
 \begin{itemize}
     \item $s(f_1)=s(f_2)$ and $r(f_1)=r(f_2)$
     \item If $\lambda\in \Lambda^B r(f_1)$, then $\lambda f_1 \sim_\Lambda \nu \mu$ implies $\mu\in \F$.
     \item If $\lambda \in s(f_1)\Lambda^B$, then $f_1\lambda \sim_\Lambda \nu \mu$ implies $\nu \in \F$. 
 \end{itemize}
	\label{stat}
\end{dfn}

\begin{lem}
\label{R-lem}
    For a $k$-graph/vertex pair, $(\Lambda,w)$, satisfying $H_R$, the proposed $\D: \Lambda \to \Z^k$ is well-defined.
\end{lem}

\begin{proof}
    For an edge $\lambda\in \Lambda^E$, $\D(\lambda)=d(\lambda)$ whenever $\lambda\notin \displaystyle{\bigcup_{x\in U_w^\E}} \Lambda^B x$. For this reason, a path in $\gamma\in G(\Lambda)^*$ will have $\D(\gamma)=d(\gamma)-\ell(\gamma) b$ where $\ell$ counts the number of co-bridge edges. For this reason, we need only prove that for paths $\gamma,\eta\in G(\Lambda)^*$ $\gamma\sim\eta$ implies $\ell(\gamma)=\ell(\eta)$.

    This can be reasoned as follows. Recall that $\sim$ is generated by transpositions of two color paths. By definition, the number of co-bridge edges will be preserved under transposition with elements of $\Lambda^\E$. Proposition \ref{prop-in-out} forbids the composition of co-bridge edges. For this reason, consider $\mu\theta$ with $\theta$ a co-bridge edge, $d(\mu)\in B$, and $s(\mu)\notin U_w^\E$. For any equivalent path $\lambda_2\lambda_1$, the sources must align so $\lambda_1$ is a co-bridge edge and $\lambda_2$ cannot be a co-bridge edge. The remaining case to check is a bridge/co-bridge pair $\theta f$. For any equivalent path $\lambda_2\lambda_1$, the stationary condition on bridge edges guarantees $\lambda_1$ a bridge edge and in particular $r(\lambda_1)\in U_w^\E$. We conclude that $\lambda_2$ is a co-bridge edge.
\end{proof}

\begin{rem}
    \label{image}
    Note that $\im(\p)$ may be partitioned into two sets, $\Xi$ and $\Theta$.
    \begin{align*}
        \Xi &= \{ \xi\in\Lambda^E : s(\xi), r(\xi)\notin U_w^\E \} \\
        \Theta &= \{ \theta f_{b,s(\theta)}: s(\theta)\in U_w^\E \text{ and } d(\theta)\in B \}
    \end{align*}
\end{rem}

\begin{cor}
    For a path $\xi\in \Xi^*$, relation $\xi\sim \zeta$ implies $\zeta\in \Xi^*$. As a consequence $\Xi^*/\sim$ is a higher rank graph. 
\end{cor}

\begin{proof}
    Suppose for sake of contradiction that $\zeta\notin \Xi^*$. Since $\xi\in \Xi^*$, we know that $s(\xi),r(\xi)\notin U_w^\E$. In particular, this means that $s(\zeta),r(\zeta)\notin U_w^\E$. Proposition \ref{prop-in-out} together with $\zeta\notin \Xi^*$ guarantee that $\ell(\zeta)>0$. This contradicts $\zeta \sim \xi$.

    Since $\Xi^*$ is closed under the equivalence relation, $\sim$, we may conclude that any contradiction of $(KG0-3)$ would contradict the fact that $(G(\Lambda)^*,\sim)$ satisfy $(KG0-3)$.
\end{proof}

\begin{lem}
\label{kg2-3-lem}
    For a $k$-graph/vertex pair, $(\Lambda,w)$, satisfying $H_R$, the triple $(\im(\p),\sim, \D)$ satisfies $(KG2-3)$.
\end{lem}

\begin{proof}
    Recall that $\im(\p)=\Xi\sqcup \Theta$. Secondly, recall that $\Xi^*/\sim$ is a well-defined higher-rank graph and $\D|_\Xi=d|_\Xi$. This means we need only check that $(KG2-3)$ are satisfied for paths containing elements of $\Theta$. 

    We begin by showing $(KG2)$. For a path $\xi \lambda f_{b,s(\lambda)}$ (which is two-colored under $\D$), there exist many paths equivalent under $\sim$, but only one such path is contained in $\im(\p)$. In particular, if $d(\xi)\in B$ then the stationary condition makes $\xi \theta f\sim \lambda$ imply that $\lambda_1=f_{a,s(\theta)}$ for some $s\in B$. If $\lambda\in\im(\p)$, then the only bridge edge available is of the type $f_{b,s(\theta)}$. This means that $c(\lambda)=d(\xi)d(\theta)b$ or $d(\theta)d(\xi)b$. Identical reasoning will show that starting with $\theta f \xi$ ensures that $c(\lambda)=d(\theta) b d(\xi)$ or $d(\xi) b d(\theta)$.

    If instead we suppose $d(\xi)\notin B$, then we claim that $\xi \theta f\sim \lambda$ with $\lambda\in \im(\p)$ implies that $d(\lambda_2)\in B$. To show this, suppose that $d(\lambda_2)\notin B$. This would imply that $d(\lambda_1)\in B$ and in particular since $s(\lambda_1)=s(f_{b,s(\theta)})$ the full reducibility of $w$ implies that $\lambda_1$ is a bridge edge. Since $d(\lambda_2)\notin B$, we conclude that $\lambda_2\in U_w^\E$ which contradicts the hypothesis that $\lambda\in \im(\p)$. The stationary condition ensures that $\lambda_3$ is not a bridge edge. Therefore $c(\lambda)= d(\xi)d(\theta)b$ or $d(\theta)bd(\xi)$.

    To demonstrate $(KG3)$, we need to consider a number of cases, yet each case uses nearly identical reasoning. For this reason we will show the case $\xi_2 \theta f \xi_1$ with $d(\xi_1)\in B$ and $d(\xi_2)\notin B$ which we believe will give the reader appropriate context to check the remaining cases.

    Consider the two distinct ordered applications of $(KG2)$. That is $\xi_2 \theta f \xi_1\sim \lambda \sim \omega \sim \zeta$ and $\xi_2 \theta f \xi_1\sim \alpha \sim \beta \sim \gamma$ 
    \begin{align*}
         c(\lambda)&=d(\xi_2)d(\xi_1)bd(\theta) & c(\alpha) &= d(\theta) b d(\xi_2) d(\xi_1) \\
        c(\omega) & = d(\xi_1)d(\xi_2)bd(\theta)& c(\beta) &= d(\theta) b d(\xi_1) d(\xi_2)  \\
        c(\zeta) & = d(\xi_1)bd(\theta)d(\xi_2) & c(\gamma) &= d(\xi_1) b d(\theta) d(\xi_2).
    \end{align*}
    Since $c(\zeta)=c(\gamma)$, $\gamma\sim \zeta$, and $\Lambda$ is a higher-rank graph we conclude that $\zeta=\gamma$ and therefore $(\im(\p),R,\sim)$ satisfies $(KG2-3)$. Lemma \ref{1-skel-realiz} ensures that $\Lambda_R$ is a well-defined $k$-graph with $\Lambda$-realization $(\p,R)$ as defined.
\end{proof}

We may now present the proof of the Theorem \ref{main-thm}

\begin{proof}[Proof of Theorem \ref{main-thm}]
    Lemmas \ref{R-lem} and $\ref{kg2-3-lem}$ combine with Proposition \ref{1-skel-realiz} to show that \textbf{(R)} is a well-defined higher-rank graph adjustment.

    It remains to show that \textbf{(R)} is subtle. For this, we suppose that $(\Lambda,w)$ is a pair satisfying $H_R$ and $\Lambda$ is a row-finite, higher-rank graph. We will then show that $C^*(\Lambda)\sim_{ME} C^*(\Lambda_R)$.

    Let $\{s_\lambda : \lambda \in \Lambda^0 \cup \Lambda^E\}$ be the canonical Cuntz-Krieger $\Lambda$-family generating $C^*(\Lambda)$. Define
	\[ T_\lambda = s_{\p(\lambda)} \text{, for } \lambda\in \Lambda_R^0 \cup \Lambda_R^E. \]
	We now proceed to show that this is a Cuntz-Krieger $\Lambda_R$-family. $(CK0)$ follows directly from $\p(x)=x$ for vertices $x$. $(CK1)$ follows from the fact that $\mu\sim_R\lambda$ if and only if $\p(\mu)\sim\p(\lambda)$. 
	
	Fix some $\mu\in \Lambda_R^E$, and it can be easily checked that $s(\p(\mu))=s_R(\mu)$. So,
	\[ T_\mu^* T_\mu = s_{\p(\mu)}^*s_{\p(\mu)} = s_{s(\p(\mu))}=s_{s_R(\mu)}=T_{s_R(\mu)}.\]
    This shows that $\{T_\lambda\}$ satisfies $(CK2)$. 
	
	All that remains is $(CK3)$. To demonstrate this, we will consider some $x\in \Lambda_R^0$ and $e_j\in \N^k$. We claim that for any $\mu\in \Lambda_R^{e_j}$ we have $T_\mu T_{\mu}^*=s_{\iota(\mu)} s_{\iota(\mu)}^*$. This amounts to recognizing first that $s_{\iota(\mu)f_{b,s(\mu)}}s^*_{\iota(\mu)f_{b,s(\mu)}} = s_{\iota(\mu)}s_{f_{b,s(\iota(\mu))}}s_{f_{b,s(\iota(\mu))}}^*s_{\iota(\mu)}$. Observe that since $s(\iota(\mu))$ is reducible with color set $B$ and $\{s_{\lambda}: \lambda \in \Lambda^0 \cup \Lambda^E\}$ satisfies $(CK3)$, we have that $s_{f_{b,s(\iota(\mu))}}s_{f_{b,s(\iota(\mu))}}^*=s_{s(\iota(\mu))}$.
	
	Consider the sum
	\[ \sum_{\mu\in x\Lambda_R^{e_j}} T_\mu T_\mu^* = \sum_{\mu\in x\Lambda_R^{e_j}}s_{\iota(\mu)}s_{\iota(\mu)}^*. \]
	
	Note that $x\notin U_w^\E$ so $\iota(x\Lambda_R^{e_j})=x\Lambda^{e_j}$. We re-index this sum to obtain
	
	\[ \sum_{\mu \in x\Lambda^{e_j}} s_{\mu}s_{\mu}^* = s_{x}=T_{x},\]
	
	and conclude that $\{T_\lambda\}$ satisfies $(CK3)$. In particular, Proposition \ref{varphi-prop} implies the existence of an injective $*$-homomorphism $\widehat{\p}: C^*(\Lambda_R)\into C^*(\Lambda)$. To demonstrate Morita equivalence, we must find a set $X\subseteq \Lambda^0$ such that $P_X C^*(\Lambda)P_X=\im(\widehat{\p})$ and demonstrate that $\Sigma(X)=\Lambda^0$. 
	
	Define the set $X=\iota(\Lambda_R^0)=\Lambda^0\setminus (U_w^\E)^0$. We claim that that $P_XC^*(\Lambda)P_X = \im(\widehat{\p})$. Note the set equality:
	\[ P_X C^*(\Lambda)P_X = \overline{\Span}\{s_\mu s_\lambda^*: r(\mu),r(\lambda) \in X \text{ and } s(\mu)=s(\lambda)\}. \]
	
	 Take $\mu,\lambda\in \Lambda$ such that $s(\mu)=s(\lambda)$ and $r(\mu),r(\lambda)\notin U_w^\E$. Since $r(\mu),r(\lambda) \notin U_w^\E$ there exists a representation of $\mu$ called $\eta \gamma$ such that $\eta\in \Xi^*$ and $|d(\eta)|$ is maximal. Likewise we take a maximal representation of $\lambda$ called $\xi \zeta$. I claim that $d(\gamma)\in\Span(B)$. If this was not the case then $c(\gamma)$ would contain a color not in $B$. By Proposition \ref{permprop}, there exists a representation of $\gamma$ with final edge color not in $B$ which by $(KG0)$ contradicts the maximality of $\eta$.

    We consider $s_\mu s_\lambda^* = s_\eta (s_\gamma s_\zeta^*) s_\xi^*$. By construction, $s_\eta,s_\xi^*\in \im(\widehat{\p})$ so we need only be concerned with $s_\gamma s_\zeta^*$. Suppose that $x:=s(\gamma)\in U_w^\E$. By the reduciblity of $s(\gamma)$ and $(CK3)$, we have $s_x=s_{f(b,x)}s_{f(b,x)}^*$. Thus we have equality $s_\gamma s_{f(b,x)}s_{f(b,x)}^* s_\zeta^*$. We may therefor assume without loss of generality that $s(\gamma)=s(\zeta)\notin U_w^\E$. 

    I claim for all $\gamma\in \Lambda$ with $d(\gamma)\in \Span(B)$ and $s(\gamma),r(\gamma)\notin U_w^\E$ $s_\gamma\in \im(\widehat{\p})$. To prove this, we rely on the stationary condition. Since bridge edges and co-bridge edges are stationary, we may factor $s_\gamma = \prod s_{\chi_i} s_{\theta_i f_i} $ with $\chi_i\in \Xi^*$ and $\theta_i f_i$ a bridge/co-bridge pair. This makes the claim equivalent to showing that $s_{\theta f}\in \im(\widehat{\p})$ for any bridge/co-bridge pair. 

    Let $\theta f$ be a bridge/co-bridge pair. If either $d(\theta)=b$ or $d(f)=b$ then $\theta f$ is equivalent to an element of $\Theta$ and $s_{\theta f}\in \im(\widehat{\p})$. Instead suppose that $d(f),d(\theta)\neq b$. Appealing to the subtlety of sink deletion and the symmetry of $\sim_{ME}$, we may assume that $\Lambda^b s(\theta)$ is nonempty (if it were not $s(\theta)$ would be a sink). Let $\vartheta\in \Lambda^b s(\theta)$. This coupled with $(CK2-3)$ yields $s_{\theta f} = (s_{\theta}s_{f(b,s(\theta))})(s_{f(b,s(\theta))}^* s_{\vartheta}^*) (s_{\vartheta} s_f)$ each of which are a bridge/co-bridge pair with an edge of degree $b$.

    This proves the claim which finishes the proof of equality $\im(\widehat{\p})=P_X C^*(\Lambda) P_X$.
	
	Finally we demonstrate that $\Sigma(X)=\Lambda^0$. This follows from the fact that $\Sigma(X)$ is hereditary. Let $y \in (U_w^\E)^0$. Since $y$ is a reducible vertex we know that $r(s_B^{-1}(y))=x\notin U_w^\E$, which allows us to conclude that $y\in \Sigma(X)$. By Theorem \ref{morita}, we conclude that $\im(\widehat{\p})\sim_{ME} C^*(\Lambda)$.
	
	Since $\widehat{\p}: C^*(\Lambda_R) \into C^*(\Lambda)$ was an isomorphism onto its image we may conclude that $C^*(\Lambda)\sim_{ME}C^*(\Lambda_R)$ and \textbf{(R)} is subtle.\qedhere

\end{proof}

As mentioned in Section \ref{moves}, it is beneficial for subtle graph moves to be robust (Definition \ref{robdef}). We conclude this section by showing that \textbf{(R)} is a robust higher-rank graph adjustment.

\begin{thm}
\label{robust}
    The move \textbf{(R)} is robust.
\end{thm}

\begin{proof}
    Suppose that $\Lambda$ and $\Omega$ are source free, row-finite graphs of rank $k$ and $j$ respectively. Further suppose that $(\Lambda,w)$ is a pair satisfying $H_R$. Our goal is two-fold. First, we show that for any $y\in \Omega^0$ the pair $(\Lambda\times\Omega, (w,y))$ satisfies $H_R$. Second we demonstrate the existence of the connecting isomorphism.

    We will use the canonical isomorphism $\N^k\times\N^j\cong \N^{k+j}$ and slight abuse of notation to partition the standard basis $E_{k+j}=E_k\sqcup E_j$. Since $(\Lambda,w)$ satisfies $H_R$, there exists a color set $B\subseteq E_k$ such that $w$ is fully reducible in color set $B$. Since $E_k\sqcup E_j \setminus B = \E \sqcup E_j$ and $G(\Lambda\times\Omega)=G(\Lambda)\square G(\Omega)$, it can be quickly checked that for any $y\in \Omega^0$ the neighborhood of color set $\E\sqcup E_j$ about $(w,y)$ is $U_w^\E \square G(\Omega)$. 

    Let $(x,y)$ be a vertex in $U_w^\E \square G(\Omega)$. Since $B\subseteq E_k$, we have $(\Lambda\times\Omega)^B (x,y) = \Lambda^Bx \times \{y\}$ and $(x,y)(\Lambda\times\Omega)^B  = \Lambda^Bx \times \{y\}$. Thus $(x,y)$ satisfies Definition \ref{red}. Since $(x,y)$ was arbitrary, we conclude that $(w,y)$ is fully reducible with color set $B$. According to Remark \ref{prod-rem}, the equivalence relation on $G(\Lambda)\square G(\Omega)$ takes the natural form $\sim_\Lambda \times \sim_\Omega$. This means that since bridge edges are stationary with color set $B$ under $\sim_\Lambda$ the set $\Lambda^B x \times \{y\}$ is stationary with color set $B$. We conclude that $(\Lambda\times \Omega, (w,y))$ satisfies $H_R$. 

    With this knowledge we may construct $\Lambda_R \times \Omega$ and $(\Lambda\times \Omega)_R$. Since \textbf{(R)} is an adjustment each will come equipped with an $\Lambda\times\Omega$-realization. Further, by the construction of \textbf{(R)} these realizations are injective.
    \begin{center}
        \begin{tikzcd}
            \Lambda \times \Omega & \Lambda_R \times \Omega \ar[l,"\p\times \id"] \\
            (\Lambda \times \Omega)_R \ar[u,"\p_\times"]
        \end{tikzcd}
    \end{center}
    The connecting isomorphism can be obtained almost immediately. Careful examination of the definition of $\p_\times$ will show that $\im(\p_\times)=\im(\p)\times \Omega$. This stems from the fact that $\p$ only acts differently from $\id$ on edges with color in the set $B$. These edges can only be found in the $\Lambda$ factor.

    From here we leverage injectivity. Since $\p_\times$ is injective there is a two sided inverse defined on $\im(\p_\times) = \im(\p)\times \Omega$. That is the map $\p_\times^{-1} \circ (\p \times \id): \Lambda_R \times \Omega \to (\Lambda \times \Omega)_R$ is a bijection that makes the diagram above commute. We conclude that \textbf{(R)} is robust.
    
\end{proof}

\section{Complete Edge Reduction, Delay, and Robustness}

As put forth in the introduction, this section will be dedicated to showing that two of the moves from \cite{efgggp} are indeed special cases of this new form of reduction. Specifically, they lie on opposite sides of a spectrum determined by the size of the color set $B$.  

We begin with the move \textbf{(CR)} which requires the definition of a \textit{complete edge}. This was a significant source of inspiration when defining stationary sets.

\begin{dfn}
	\cite[Notation 6.1]{efgggp} We say a collection of edges, $G\subseteq \Lambda^E$, is a \textit{complete edge} if it has the following properties:
	\begin{enumerate}[(1)]
		\item $G$ contains precisely one edge of each color;
		\item $s(g_1)=s(g_2)$ and $r(g_1)=r(g_2)$ for every $g_1,g_2\in G$;
		\item if $g_1\in G$ and $\mu,\eta, g\in \Lambda^E$ satisfy $\mu g_1\sim g\eta$ or $\mu g_1\sim \eta g$, then $g\in G$.
	\end{enumerate}
\end{dfn}

One can check that a complete edge is indeed a stationary set with color set $E$.

\begin{dfn}
	\textbf{Complete Edge Reduction (CR)\cite[Definition 6.3]{efgggp}: } Let $\Lambda$ be a $k$-graph and fix $w\in \Lambda^0$ such that $\Lambda^E w$ and $w\Lambda^E$ are complete edges and $w\neq r(\Lambda^E w)=:x$. Define
	\begin{align*}
		\Lambda_{CR}^0 &= \Lambda^0 \setminus \{w\}\\
		\Lambda_{CR}^E & = \Lambda^E\setminus \Lambda^E w \\
		s_{CR}(e)&=s(e)\\
		r_{CR}(e)& =\begin{cases}
			r(e), & r(e)\neq w \\
			x, & r(e)=w
		\end{cases}
	\end{align*}
	Then fix $g\in \Lambda^E w$ and define $\p$ on $y\in \Lambda_{CR}^0$ and $\gamma\in \Lambda_{CR}^E$ in the following way:
	\begin{align*}
		\p(y) & = \iota(y) \\
		\p(\gamma) & = \begin{cases}
			\iota(\gamma), & r(\iota(\gamma))\neq w \\
			g \iota(\gamma), & r(\iota(\gamma))=w.
		\end{cases}
	\end{align*}
	Now extend multiplicatively and define $\mu \sim_{CR} \lambda$ if and only if $\p(\mu)\sim \p(\lambda)$.
\end{dfn}

This move shares enough in common with \textbf{(R)} that we can directly demonstrate $\Lambda_R \cong \Lambda_{CR}$. 

\begin{prop}
	If $(\Lambda,w)$ is a pair satisfying $H_{CR}$, then the vertex $w$ is fully reducible with color set $E$, and $\Lambda^Ew=\{f_{b}:b\in B\}$ is a stationary set. Since $(\Lambda,w)$ satisfies $H_R$ and $H_{CR}$ $\Lambda_R$ and $\Lambda_{CR}$ exist with $\Lambda_{CR}\cong\Lambda_R$. 
\end{prop}

\begin{proof}
	Since $w\Lambda^E$ is a complete edge, by definition for each $e_i$ we have $w\Lambda^E=\{f_{e_i}\}$. Since $x=r(\Lambda^Ew)\neq w$, we quickly conclude that $w$ is reducible with color set $E$. Since $\E=\emptyset$, we know that $U_w^\E=\{w\}$ and thus we have that $w$ is fully reducible with color set $E$. We note that since $w\Lambda^E$ is a complete edge, $\{f_{b,w}:b\in B\}$ is a stationary set with color set $E$. 
	
	We may therefore construct $\Lambda_R$. For ease of notation, we will label $\Lambda^Ew=\{g_b:b\in B\}$. Notice first that $\Lambda_{CR}^0=\Lambda_R^0$ since $U_{w}^\E=\{w\}$. Compare $\Lambda_{CR}^E=\Lambda^E\setminus \Lambda^Ew$ and $\Lambda_R^E = \Lambda^E \setminus r^{-1}(w)$, and note that $r^{-1}(w)=\{f_{b,w}:b\in B\}$ and $\Lambda^Ew=\{g_b:b\in B\}$. We will associate $g_b\mapsto f_{b,w}$ and see that this yeilds a bijection between $\Lambda_{CR}^E$ and $\Lambda_R^E$. We need to check that $s_R, s_{CR}, r_R,$ and $r_{CR}$ all agree under this bijection.  Observe that  
	\[ s_R(g_b) = s(f_{b,w}) = s_{CR}(f_{b,w}) \qquad \text{ and } \qquad  r_R(g_b) = r(g_b) = x = r_{CR}(f_{b,w}). \]
	We now check that $\sim_R$ gives the same equivalence classes as $\sim_{CR}$. This amounts to checking that after fixing $g\in \Lambda^Ew$ for $\p_{CR}$ and $d(g)\in E$ for $\p_R$, we have $\p_R(g_b)= g_bf_{d(g),w}\sim g f_{b,w} = \p_{CR}(f_{b,w})$. We may now conclude that $\Lambda_{CR}\cong\Lambda_R$.
\end{proof}

To discuss the move Delay \textbf{(D)} and its relationship to \textbf{(R)}, we first present \cite[Definition 4.1]{efgggp}, with some notation changes. This definition is particularly technical. In essence, the goal is to choose an edge $f\in \Lambda^{e_1}$ and replace it with a path of length $2$, $f^2f^1$. However, as we saw with reduction, local changes like this have a global effect on the $k$-graph. To account for this, \textbf{(D)} adds edges $C^{e_i}_D$ for each $k\geq i>1$ and also delays the edges in the set $C^{e_1}$. To prime the reader for this definition, we will state the following proposition which will be proved after the definition.

\begin{prop}
\label{propdel}
    The pair $(\Lambda_D, v_f)$ satisfy $H_R$. In particular,
\begin{enumerate}[(I)]
    \item $r_D(f^1)=:v_f$ is reducible with color set $\{e_1\}$.
    \item $U_{v_f}^\E=\displaystyle{\bigcup_{i=2}^k C_D^{e_i}}$.
    \item $(U_{v_f}^\E)^0 = \{ v_g : g\in C^{e_1} \}$ and each of these are reducible with color set $\{e_1\}$.
    \item $v_g\Lambda^B=\{ g^1 \}$ and thus is stationary with color set $\{e_1\}$.
\end{enumerate}
\end{prop}

\begin{dfn}
\label{def:delay}
\textbf{Delay (D)\cite[Definition 4.1]{efgggp}:} Let $(\Lambda, d)$ be a $k$-graph and $G = (\Lambda^0, \Lambda^E, r, s)$ its underlying directed graph.  Fix $f \in \Lambda^E$; without loss of generality, assume $d(f) = e_1$.   Define the sets:
\begin{align*}
A_1 &= \{ f \} \cup  \{ g \in \Lambda^{e_1} : ag \sim fb \text{ or } ga \sim bf  \text{ for some } a,b \in \Lambda^{e_i} \text{ for }  2 \leq i \leq k   \}, \\
A_m &= \{ e \in \Lambda^{e_1} : ag \sim eb \text{ or } ga \sim be
 \text{ where } a,b \in \Lambda^{e_i} \text{ for }  2 \leq i \leq k , ~ g \in A_{m-1} \}, \\
C^{e_1} &= \bigcup_{j = 1}^{\infty} A_j  \subseteq \Lambda^{e_1}. 
\end{align*}
\noindent
Define $C^{e_i}:$ 
\[
C^{e_i} = \{ [ga] \in \Lambda : g \in C^{e_1}, a \in \Lambda^{e_i} \} . 
\]
From these sets in $\Lambda$ we define the sets to be added to form $\Lambda_D$. 

\[
C_D^{e_1} = \{ g^1, g^2 : g \in C^{e_1} \} 
\]

\[
C_D^{e_i} = \{ e_{\alpha} : \alpha \in C^{e_i} \}.
\]
Define the $k$-colored graph $G_D = (\Lambda_D^0, \Lambda_D^1, r_D, s_D)$ by
\begin{equation*}
\begin{split}
\Lambda_D^0 &= \Lambda^0 \cup \{ v_g \}_{g \in C^{e_1}},  
\quad \Lambda_D^{e_1} = (\Lambda^{e_1} \setminus C^{e_1}) \cup C_D^{e_1} \text{, with} \\ 
& \hspace{10mm} s_D(g^1) = s(g), s_D(g^2) = v_g, 
 \hspace{3mm} r_D(g^1) = v_g, r_D(g^2) = r(g) ; \\ 
\Lambda_D^{e_i}& = \Lambda^{e_i} \cup C_D^{e_i} \text{, with} \\ 
& \hspace{10mm} s_D(e_{\alpha}) = v_{g} \text{ such that } bg \text{ represents } \alpha \text{ and } d(g) = e_1, \\
& \hspace{10mm} r_D(e_{\alpha}) = v_{h} \text{ such that } ha \text{ represents } \alpha \text{ and } d(h) = e_1.
\end{split}
\end{equation*}

Let $\iota_D : G_D \to G$ be the partially defined inclusion map with domain $(\Lambda_D^0 \cup \Lambda_D^1) \setminus \left(\{\bigcup\limits_{i=1}^k C_D^{e_i}\} \cup \{v_e : e \in C^{e_1}\}\right)$. 
Then, for edges $ g\in \Lambda^1_D \backslash \displaystyle{\bigcup_{i=1}^k} C_D^{e_i}$, we can define
\begin{equation*}
s_D(g) = s(\iota_D(g)), ~r_D(g) = r(\iota(g)), ~d_D(g) = d(\iota(g)) .
\end{equation*}

Let $G^*_D$ be the path category for $G_D$ and define the equivalence relation $\sim_D$ on bi-color paths  $\mu = \mu_2 \mu_1  \in G_D^2$ according to the following rules. 
\\

Case 1: Assume $\mu_1, \mu_2 \notin \bigcup_{i=1}^k C_D^{e_i}$. Then we set $[\mu]_D = \iota^{-1}([\iota(\mu)])$.
\\

Case 2: Suppose  $\mu_j$ lies in $ C^{e_1}_D$, so that $\mu_j \in \{ g^1, g^2\}$ for some edge $g \in C^{e_1}$.  If $j=1$ and $\mu_1 = g^2$, then $r(\mu_1) = s(\mu_2) = \iota^{-1}( r(g)) \in \iota^{-1}(\Lambda^0)$, and the edges in $G_D$ with source in $\iota^{-1}(\Lambda^0)$ and degree $e_i$ for $i \not= 1$  are in $\iota^{-1}(\Lambda^1)$.  Therefore $\mu_2 \in \iota^{-1}(\Lambda^{e_i})$, and $\iota(\mu_2) g$ is a bi-color path in $G$, so $\iota(\mu_2) g \sim h a$ for edges $h \in C^{e_1}, a \in \Lambda^{e_i}$.  There is then an edge $e_{[\mu_2 g]} \in \Lambda^{e_i}_D$ with source $s(\mu_1) = v_g$ and range $v_h = s(h^2)$;  we define $\mu_2 \mu_1 =  \mu_2 g^2  \sim_D h^2 e_{[\mu_2 g]}$.
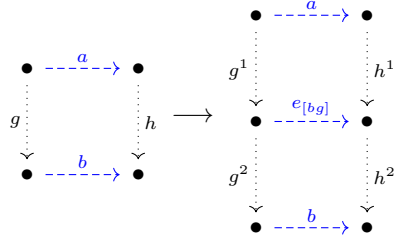
\begin{figure}[H]
\stepcounter{thm}
\begin{center}
\begin{tikzcd}[column sep=1cm,row sep=1cm]
\bullet \ar[r, dashed, blue, "a"] \ar[d, dotted, "g"'] & \bullet \ar[d, dotted, "h"]\\
\bullet \ar[r, dashed, blue, "b"] & \bullet
\end{tikzcd}
$\longrightarrow$
\begin{tikzcd}[column sep=1cm,row sep=1cm]
\bullet \ar[r, dashed, blue, "a"] \ar[d, dotted, "g^1"'] & \bullet \ar[d, dotted, "h^1"]\\
\bullet \ar[r, dashed, blue, "e_{[bg]}"] \ar[d, dotted, "g^2"'] & \bullet \ar[d, dotted, "h^2"]\\
\bullet \ar[r, dashed, blue, "b"] & \bullet
\end{tikzcd}
\end{center}
    \caption{A commuting square in $G$ and its ``children" in $G_D$, when $h,g \in C^{e_1}$.}
    \label{fig:Case2DelayDef}
\end{figure}
If $j =1$ and $\mu_1 = g^1$, the only edges in $G_D$ with source $r(g^1) = v_g $ and degree $e_i$ for $i \not=1$ are of the form $e_{[bg ]} = e_{[ha]}$ for some commuting square $bg \sim ha $ in $ \Lambda$.  In this case, we will have $h \in C^{e_1}$, and $r(h^1) = v_h = r(e_{[bg]})$, so we set $e_{[bg]} g^1 \sim_D  h^1 a$.

A similar argument shows that if $j=2$, the path $\mu_2 \mu_1$ will be of the form $h^1 a$ or $h^2 e_{[ha]}$, whose factorizations we have already described.
\\

Case 3: Assume $\mu$ is of the form $e_{\beta} e_{\alpha}$ for $\alpha \in C_D^{e_i}$, and $\beta \in C_D^{e_j}$ with $i \neq j$. Now $s_D(e_{\beta}) = r_D(e_{\alpha}) = v_g$ for some $g \in C^{e_1}$, and consequently $\alpha, \beta \in \Lambda$ are linked as shown on the left of Figure \ref{fig:Case3DelayDef}. Since $\Lambda$ is a $k$-graph, the 3-color path outlining $\beta \alpha$ generates a 3-cube in $\Lambda$, which is depicted on the right of Figure \ref{fig:Case3DelayDef}. 
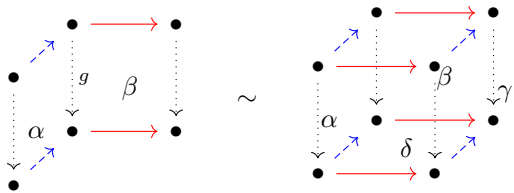
\begin{figure}[H]
\stepcounter{thm}
\begin{center}
\begin{tikzcd}[column sep=.3cm,row sep=.3cm]
	& \bullet \ar[rr, red] \ar[dd, black, dotted, "g"] & \ar[ddr,phantom,shift right=2ex, "\beta"]& \bullet \ar[dd, black, dotted]\\
	\bullet \ar[dd, black, dotted] \ar[dd,phantom, shift left =2ex, "\alpha"] \ar[ur, blue, dashed] & & & \\
	& \bullet \ar[rr, red] & & \bullet \\
	\bullet \ar[ru, blue, dashed] & &
\end{tikzcd} $~~~\sim ~~~$
\begin{tikzcd}[column sep=.3cm,row sep=.3cm]
	& \bullet \ar[rr, red] \ar[dd, black, dotted] &  \ar[ddr,phantom,shift right=2ex, "\beta"]& \bullet \ar[dd, black, dotted] \\
	\bullet \ar[dd, black, dotted] \ar[ru, blue, dashed] \ar[rr, red] \ar[dd, phantom, shift left=1ex, "\alpha"] &  & \bullet \ar[ru, blue, dashed]  \ar[dd, black, dotted] & \ar[d, phantom, shift left = 1ex, "\gamma"'] \\
	& \bullet \ar[rr, red] \ar[dr, phantom,  "\delta"] & & \bullet \\
	\bullet \ar[ru, blue, dashed] \ar[rr, red] & & \bullet \ar[ru, blue, dashed] & \\
\end{tikzcd}
    \caption{The commuting squares of edges from $\bigcup_{i=2}^k C_D^{e_i}$.}
    \label{fig:Case3DelayDef}
    
\end{center}
\end{figure}
Let $\delta$ and $\gamma$ denote the faces of this cube which lie, respectively, opposite $\beta$ and $\alpha$. Since $g \in C^{e_1}$,  all of the vertical edges of this cube are in $C^{e_1},$ and so $\delta \in C^{e_j}, \gamma \in C^{e_i}.$ Moreover, the path $e_{\gamma} e_{\delta}$ is composable in $\Lambda_D$, and has the same source and range as $e_{\beta} e_{\alpha}$.  Set $e_{\beta} e_{\alpha} \sim_D e_{\gamma} e_{\delta}$.

\end{dfn}

Before proving Proposition \ref{propdel} it is important to point out that \cite[Theorem 4.2, 4.3]{efgggp} showed that \textbf{(D)} is a well-defined, subtle,  higher-rank graph adjustment.

\begin{proof}[Proof of Proposition \ref{propdel}]
    In this proof, all edges and vertices are in $\Lambda_D$. So, for notation's sake we will write $s$ and $r$ instead of $s_D$ and $r_D$.
    \begin{enumerate}[(I)]
        \item This will follow directly from the definitions of $s_D$ and $r_D$ put forth above. Observe that $v_f \Lambda^E_D s(f) = \{f_1\}$ and this edge has color $e_1$ ensuring that the bridge edges are $f^1$ and the co-bridge edges are $f^2$.
        \item We first consider $\Lambda_D^\E$. This removes the edges in $C^{e_1}_D$. From the given definitions of $s_D$ and $r_D$, these are the only edges with source $v_g$, for some $g\in C^{e_1}$, and range in $\Lambda$. Further, every edge of the undirected connected component is contained in $C^{e_i}_D$ for some $1<i\leq k$, since the inductive definition is inherently undirected. This returns the desired equality. 
        \item Suppose that $g\in C^{e_1}$. By definition, there exists some edge $a$ with $d(a)\neq e_1$ such that $ga$ is a well-defined path. In particular,  $[ga]\in C^{d(a)}$. So $e_{[ga]}\in C^{d(a)}_D$ and further $r(e_{[ga]})=v_g$ and $s(e_{[ga]})=v_h$ for some $h\in C^{e_1}$. We conclude that $(U_{v_f}^\E)^0 = s\left( U_{v_f}^\E \right) \cup r\left( U_{v_f}^\E \right) = \{v_g: g\in C^{e_1}\}$. From here we need only notice that a given $v_g$ is reducible with color set $\{e_1\}$. In particular, $v_g$ has bridge edge $g^1$ and co-bridge edge $g^2$ allowing us to conclude full reducibility of $v_f$.
        \item This follows from the argument above. Further note from the definition of stationary that a singleton edge connecting $s(g)$ and $v_g$ is vacuously stationary. 
 
    \end{enumerate}
\end{proof}

With \textbf{(D)} it is clear that objects are added to the $k$-graph while \textbf{(R)} removes edges and vertices. This means that its relationship with \textbf{(R)} is clearly not a simple matter of $\Lambda_D\cong\Lambda_R$. Instead we recall that in \cite{compclass} the sequences of graph moves between $E$ and $F$ can utilize these moves or their inverses (i.e. the path between $E$ and $F$ is possibly undirected). This means that if we show $(\Lambda_D)_R \cong \Lambda$, then we may conclude that the Morita equivalence between $C^*(\Lambda_D) \sim_{ME} C^*(\Lambda)$ can be thought of as a direct consequence of the fact that $C^*((\Lambda_D)_R) \sim_{ME} C^*(\Lambda_D)$. To this end we present the following proposition.

\begin{prop}
	For any $k$-graph, $\Lambda$, if $\Lambda_D$ is obtained by delaying $f\in \Lambda^{e_1}$ and reduction is performed on $(\Lambda_D, v_f)$ to obtain $(\Lambda_{D})_R$, we have $(\Lambda_D)_R \cong \Lambda$.
\end{prop}

\begin{proof}
    From Proposition \ref{propdel} we know that $(\Lambda_D)_R$ is well-defined. Further we may note that 
    
    \begin{align*}
        (\Lambda_D)_R^0 &= \Lambda_D^0 \setminus \{v_g: g\in C^{e_1}\} = (\Lambda^0\cup\{v_g: g\in C^{e_1}\}) \setminus \{v_g: g\in C^{e_1}\} = \Lambda^0\\
        (\Lambda_D)_R^E &= \Lambda_D^E \setminus ((U_{v_f}^\E)^E\cup\{g^1:g\in C^{e_1}\}) = \Lambda_D^E \setminus \left(\left( \displaystyle{\bigcup_{i=2}^k C_D^{e_i}} \right)\cup\{g^1:g\in C^{e_1}\}\right)\\
        & = (\Lambda^E \setminus \{g\in C^{e_1}\})\cup \{g^2: g\in C^{e_1}\}
    \end{align*} 
    
    At first it is worrying that we do not recover $\Lambda^E$ exactly, but this can be easily remedied by the identification $g^2\mapsto g$. This is the right thing to do, because $r_R(g^2)=r_D(\iota_R(g^2)) = r_D(g^2)=r(g)$. Recall also that $g^2\in (\Lambda_{D})_R$ is such that $s_D(\iota_R(g^2))=v_g$ and thus $s_{R}(g^2)=s_D(g^2g^1)=s(g)$. 
    
    Being more precise, there is a bijective functor between the path categories $\phi: G((\Lambda_D)_R)^*\to G(\Lambda)^*$ such that for all $g^2\in C^{e_1}_D$ we have $\phi(\iota_R^{-1}(g^2))= g$ and, for the remaining edges, $\mu\mapsto\iota_D(\iota_R(\mu))$. It remains to show that this identification respects the relations $\sim_{DR}$ and $\sim$. That is, given $\mu\sim_{DR}\nu$ then $\phi(\mu)\sim \phi(\nu)$. This proof relies primarily on unpacking the definition of $\sim_{DR}$ and observing that $\phi$ acts this way by construction. 
    
    Using the notation from Remark \ref{image}, we may observe that $\Theta = \{ g^2g^1 : g\in C^{e_1} \}$. In particular this means that 
    
    \[ \p(\mu)=\begin{cases} g^2g^1, & \iota_R(\mu)=g^2 \text{ for some $g\in C^{e_1}$}\\ \iota_R(\mu), & \text{otherwise.} \end{cases} \]
    
    Let's consider some path of length $2$, $\mu =\mu_2\mu_1$, in $(\Lambda_{D})_R$ and observe that $[\mu]_{DR} = [\p(\mu)]_D$. We will investigate cases based on whether $\p(\mu)\in \Xi^*$ or $\p(\mu)\notin \Xi^*$.
    
    Suppose that $\p(\mu)\in \Xi^*$. Then $\p(\mu)=\iota_R(\mu)$. Further, since $\p(\mu)\in \Xi^*$, we conclude that $s_D(\iota_R(\mu_1)), s_D(\iota_R(\mu_2)) \notin \{ v_g: g\in C^{e_1}\}$ and thus we fall under \textit{Case 1} for $\sim_D$. We conclude that:
    \[ [\mu]_{DR} = [\p(\mu)]_D = [\iota_R(\mu)]_D = [\iota_D(\iota_R(\mu))]=[\phi(\mu)]. \]
	
	Now suppose that either $\p(\mu_1)\in \Theta$ or $\p(\mu_2)\in \Theta$. We will closely examine the former, and note that a similar argument holds for the latter. If $\p(\mu_1)\in \Theta$ then $\mu_1=\iota^{-1}_R(g^2)$ for some $g\in C^{e_1}$ and $\p(\mu_2)=\iota_R(\mu_2)$. So, $\p(\mu)=g^2g^1\iota_R(\mu_2)$. By the construction of \textit{Case 2} we know that
	
	$$[\mu_2\mu_1]_{DR}=[\iota_R(\mu) g^2g^1]_D = [\iota_D(\iota_R(\mu_2)) g] = [\phi(\mu_2) \phi(\mu_1))] .$$
	
	We conclude that $\phi$ induces an isomorphism of $k$-graphs $\Lambda\cong (\Lambda_{D})_R$.
	
\end{proof}

From the above results we see that the moves \textbf{(CR)} and \textbf{(D)} from \cite{efgggp} are special cases of the move \textbf{(R)} developed in this paper. This allows us to update the list of necessary Morita equivalence preserving $k$-graph moves. 
\begin{multicols}{2}

\begin{itemize}
    \item Insplitting \textbf{(I)} \cite[\S 3]{efgggp}
    \item Outsplitting \textbf{(O)} \cite[\S 4]{outsplit}
\end{itemize}
\begin{itemize}
    \item Sink Deletion \textbf{(S)} \cite[\S 6]{efgggp}
    \item Reduction \textbf{(R)}
\end{itemize}
\end{multicols}

\bibliography{The}{}

\newcommand{\etalchar}[1]{$^{#1}$}
\begin{thebibliography}{HRSW13}

\bibitem[All08]{allen}
Stephen Allen.
\newblock A gauge invariant uniqueness theorem for corners of higher rank graph
  algebras.
\newblock {\em Rocky Mountain J. Math.}, 38(6):1887--1907, 2008.

\bibitem[BGR77]{morita-stab}
Lawrence Brown, Philip Green, and Marc Rieffel.
\newblock Stable isomorphism and strong morita equivalence.
\newblock {\em Pacific Journal of Mathematics}, 71(2):349--363, 1977.

\bibitem[BP04]{delay}
Teresa Bates and David Pask.
\newblock Flow equivalence of graph algebras.
\newblock {\em Ergodic Theory Dynam. Systems}, 24(2):367--382, 2004.

\bibitem[Cun81]{cuntz}
J.~Cuntz.
\newblock A class of {$C\sp{\ast} $}-algebras and topological {M}arkov chains.
  {II}. {R}educible chains and the {E}xt-functor for {$C\sp{\ast} $}-algebras.
\newblock {\em Invent. Math.}, 63(1):25--40, 1981.

\bibitem[EFG{\etalchar{+}}22]{efgggp}
Caleb Eckhardt, Kit Fieldhouse, Daniel Gent, Elizabeth Gillaspy, Ian Gonzales,
  and David Pask.
\newblock Moves on {$k$}-graphs preserving {M}orita equivalence.
\newblock {\em Canad. J. Math.}, 74(3):655--685, 2022.

\bibitem[EGLN20]{classification-KK-cont}
George~A. Elliott, Guihua Gong, Huaxin Lin, and Zhuang Niu.
\newblock The classification of simple separable {KK}-contractible {$\rm
  C^*$}-algebras with finite nuclear dimension.
\newblock {\em J. Geom. Phys.}, 158:103861, 51, 2020.

\bibitem[ERRS21]{compclass}
S{\o}ren Eilers, Gunnar Restorff, Efren Ruiz, and Adam P.~W. S{\o}rensen.
\newblock The complete classification of unital graph {$C^*$}-algebras:
  geometric and strong.
\newblock {\em Duke Math. J.}, 170(11):2421--2517, 2021.

\bibitem[Eva08]{evans08}
D.~Gwion Evans.
\newblock On the {$K$}-theory of higher rank graph {$C^*$}-algebras.
\newblock {\em New York J. Math.}, 14:1--31, 2008.

\bibitem[GLN20a]{finite-simple-amen-1}
Guihua Gong, Huaxin Lin, and Zhuang Niu.
\newblock A classification of finite simple amenable {$\mathcal{Z}$}-stable
  {$C^\ast$}-algebras, {I}: {$C^\ast$}-algebras with generalized tracial rank
  one.
\newblock {\em C. R. Math. Acad. Sci. Soc. R. Can.}, 42(3):63--450, 2020.

\bibitem[GLN20b]{finite-simple-amen-2}
Guihua Gong, Huaxin Lin, and Zhuang Niu.
\newblock A classification of finite simple amenable {$\mathcal{Z}$}-stable
  {${\rm C}^\ast$}-algebras, {II}: {${\rm C}^\ast$}-algebras with rational
  generalized tracial rank one.
\newblock {\em C. R. Math. Acad. Sci. Soc. R. Can.}, 42(4):451--539, 2020.

\bibitem[HNP{\etalchar{+}}18]{hajac-sims}
Piotr~M. Hajac, Ryszard Nest, David Pask, Aidan Sims, and Bartosz
  Zieli\'{n}ski.
\newblock The {$K$}-theory of twisted multipullback quantum odd spheres and
  complex projective spaces.
\newblock {\em J. Noncommut. Geom.}, 12(3):823--863, 2018.

\bibitem[HRSW13]{oneskel}
Robert Hazlewood, Iain Raeburn, Aidan Sims, and Samuel B.~G. Webster.
\newblock Remarks on some fundamental results about higher-rank graphs and
  their {$C^*$}-algebras.
\newblock {\em Proc. Edinb. Math. Soc. (2)}, 56(2):575--597, 2013.

\bibitem[KP00]{kp}
Alex Kumjian and David Pask.
\newblock Higher rank graph {$C^\ast$}-algebras.
\newblock {\em New York J. Math.}, 6:1--20, 2000.

\bibitem[KPS11]{quasimorph1}
Alex Kumjian, David Pask, and Aidan Sims.
\newblock Generalised morphisms of {$k$}-graphs: {$k$}-morphs.
\newblock {\em Trans. Amer. Math. Soc.}, 363(5):2599--2626, 2011.

\bibitem[Lisar]{outsplit}
Ben Listhartke.
\newblock {\em Outsplits of k-Graphs}.
\newblock PhD thesis, Kansas State University, To Appear.

\bibitem[PRS08]{pask-rennie-sims}
David Pask, Adam Rennie, and Aidan Sims.
\newblock The noncommutative geometry of {$k$}-graph {$C^*$}-algebras.
\newblock {\em J. K-Theory}, 1(2):259--304, 2008.

\bibitem[RS87]{rose}
Jonathan Rosenberg and Claude Schochet.
\newblock The {K}\"{u}nneth theorem and the universal coefficient theorem for
  {K}asparov's generalized {$K$}-functor.
\newblock {\em Duke Math. J.}, 55(2):431--474, 1987.

\bibitem[RS04]{rae}
Iain Raeburn and Wojciech Szyma\'{n}ski.
\newblock Cuntz-{K}rieger algebras of infinite graphs and matrices.
\newblock {\em Trans. Amer. Math. Soc.}, 356(1):39--59, 2004.

\bibitem[RSS15]{ruiz-sims-sorensen}
Efren Ruiz, Aidan Sims, and Adam P.~W. S{\o}rensen.
\newblock U{CT}-{K}irchberg algebras have nuclear dimension one.
\newblock {\em Adv. Math.}, 279:1--28, 2015.

\bibitem[S{\o}r13]{move-eq}
Adam P.~W. S{\o}rensen.
\newblock Geometric classification of simple graph algebras.
\newblock {\em Ergodic Theory Dynam. Systems}, 33(4):1199--1220, 2013.

\bibitem[TWW17]{quasidiag}
Aaron Tikuisis, Stuart White, and Wilhelm Winter.
\newblock Quasidiagonality of nuclear {$C^\ast$}-algebras.
\newblock {\em Ann. of Math. (2)}, 185(1):229--284, 2017.

\bibitem[Vdo21]{vdovina-DM-solns}
Alina Vdovina.
\newblock Drinfeld-{M}anin solutions of the {Y}ang-{B}axter equation coming
  from cube complexes.
\newblock {\em Internat. J. Algebra Comput.}, 31(4):775--788, 2021.

\bibitem[Yan16]{yang}
Dilian Yang.
\newblock The interplay between {$k$}-graphs and the {Y}ang-{B}axter equation.
\newblock {\em J. Algebra}, 451:494--525, 2016.

\end{thebibliography}
\bibliographystyle{alpha}

\end{document}